\documentclass[11pt]{article}

\usepackage{amsmath, amsfonts, amssymb, amsthm,  graphicx, enumerate}
\usepackage[letterpaper, left=1truein, right=1truein, top = 1truein, bottom = 1truein]{geometry}

\usepackage[numbers, square]{natbib}

\usepackage[
            CJKbookmarks=true,
            bookmarksnumbered=true,
            bookmarksopen=true,
            colorlinks=true,
            citecolor=red,
            linkcolor=blue,
            anchorcolor=red,
            urlcolor=blue,
            ]{hyperref}

\usepackage[usenames,dvipsnames]{color}

\usepackage[ruled, vlined, lined, commentsnumbered]{algorithm2e}

\usepackage{prettyref,soul,xspace}

\usepackage{tikz}
\usepackage{pgfplots,makecell}

\usepackage{bm}
\usepackage{fancyhdr}
\theoremstyle{plain}
\newtheorem{theorem}{Theorem}[section]
\newtheorem{lemma}[theorem]{Lemma}
\newtheorem{proposition}[theorem]{Proposition}
\newtheorem{corollary}[theorem]{Corollary}

\newcommand{\p}{\mathbb{P}}
\newcommand{\E}{\mathbb{E}}

\newcommand{\br}[1]{\left( #1 \right)}
\newcommand{\sbr}[1]{\left[ #1 \right]}

\newcommand{\pbr}[1]{\p\left( #1 \right)}
\newcommand{\ebr}[1]{\exp\left( #1 \right)}
\newcommand{\abs}[1]{\left| #1 \right|}

\newcommand{\mathr}{\mathbb{R}}

\newcommand{\indic}[1]{{\mathbb{I}\left\{{#1}\right\}}}
\newcommand{\iprod}[2]{\left \langle #1, #2 \right\rangle}
\newcommand{\norm}[1]{\left\|{#1} \right\|}
\newcommand{\opnorm}[1]{\norm{#1}_{\rm op}}

\newcommand{\fnorm}[1]{\norm{#1}_{\rm F}}
\newcommand{\argmin}{\mathop{\rm argmin}}
\newcommand{\argmax}{\mathop{\rm argmax}}

\newcommand{\mathcleq}{\mathc_{\leq 1}}
\newcommand{\mathc}{\mathbb{C}}

\newcommand{\wh}{\hat}
\newcommand{\Tr}{\mathop{\rm Tr}}

\def\H{{ \mathrm{\scriptscriptstyle H} }}

\newcommand{\MLE}{\textsc{MLE}}
\newcommand{\SDP}{\textsc{SDP}}
\newcommand{\BM}{\textsc{BM}}

\renewcommand{\bar}[1]{\overline{#1}}

\newcommand{\re}{{\rm Re}}

\numberwithin{equation}{section}
\newcommand\numberthis{\addtocounter{equation}{1}\tag{\theequation}}

\newcommand{\xdownarrow}[1]{%
  {\left\updownarrow\vbox to #1{}\right.\kern-\nulldelimiterspace}
}

\title{Tightness of SDP and Burer-Monteiro Factorization for Phase Synchronization in High-Noise Regime}
\author{Anderson Ye Zhang\\
~\\
University of Pennsylvania
}

\date{}

\begin{document}
\maketitle

\begin{abstract}
We study the difference between the maximum likelihood estimation (MLE) and its semi-definite programming (SDP) relaxation for the phase synchronization problem, where $n$ latent phases are estimated based on pairwise observations corrupted by Gaussian noise at a level $\sigma$. 
While previous studies have established that SDP coincides with the MLE when $\sigma \lesssim \sqrt{n / \log n}$, the behavior in the high-noise regime $\sigma \gtrsim \sqrt{n / \log n}$ remains unclear.
We address this gap by quantifying the deviation between the SDP and the MLE in the high-noise regime as $\exp(-c \frac{n}{\sigma^2})$, indicating an exponentially small discrepancy. In fact, we establish more general results for the Burer-Monteiro (BM) factorization that covers the SDP as a special case: it has the exponentially small deviation from the MLE in the high-noise regime and coincides with the MLE when $\sigma$ is small. To obtain our results, we develop a refined entrywise analysis of the MLE that is beyond  the existing $\ell_\infty$ analysis in literature.
\end{abstract}

\section{Introduction}
In this paper, we study the phase synchronization problem \cite{singer2011angular,bandeira2017tightness,abbe2017group,zhong2018near}. Let $z_1^*,\ldots,z_n^*\in\mathc_1$ be latent parameters where $\mathc_1=\{x\in\mathc:|x|=1\}$ includes all unit complex numbers. That is, each $z^*_j$ represents an angle in $[0,2\pi)$ or a phase. The observations are 
\begin{align}\label{eqn:phase_model}
Y_{jk} = z_j^*\bar{z_k^*} + \sigma W_{jk},\quad 1\leq j<k\leq n,
\end{align}
where $\sigma>0$ is the noise level and $\{W_{jk}\}_{1\leq j<k\leq n}\in\mathc$ are the additive noises following the standard complex Gaussian distribution independently. This model can be conveniently expressed in matrix form.
By defining  $Y_{jj}=1,W_{jj}=0$ for all $j\in[n]$ and $Y_{kj}=\bar{Y_{jk}},W_{kj}=\bar{W_{jk}}$ for all $1\leq j<k\leq n$, we can rewrite the model as
\begin{align}\label{eqn:phase_model_matrix}
Y = z^*(z^*)^\H + \sigma W\in\mathc^{n\times n},
\end{align}
where $z^*\in\mathc_1^n$ with coordinates $z^*_1,\ldots,z^*_n$. The goal is to estimate the latent vector $z^*$ from the observed matrix $Y$.

To solve the phase synchronization problem, one natural approach is to use maximum likelihood estimation (MLE) \cite{gao2021exact, zhong2018near}. The MLE can be formulated as the following optimization problem:
\begin{align}\label{eqn:MLE_def}
\hat z^{\MLE} = \argmax_{z\in\mathc_1^n} \iprod{ Y}{zz^\H}. 
\end{align}
However, this optimization is over a non-convex set $\mathc_1^n$, making it computationally challenging. To overcome this computational difficulty, note that $\hat z^\MLE$ satisfies
\begin{align}\label{eqn:MLE_def_2}
\hat z^{\MLE}(\hat z^{\MLE})^\H =\argmax_{Z\in\mathc^{n\times n}: Z=Z^\H, \text{rank}(Z)=1, Z_{jj}=1,\forall j\in[n]} \iprod{Y}{Z},
\end{align}
which can be relaxed into a convex problem using semi-definite programming (SDP) \cite{arie2012global,singer2011three, ling2020solving, fan2021joint, wang2013exact, gao2022sdp, javanmard2016phase}:
\begin{align}\label{eqn:SDP_def}
\hat Z^{\SDP} = \argmax_{Z\in\mathc^{n\times n}:Z=Z^\H,Z\succeq 0,Z_{jj}=1,\forall j\in[n]} \iprod{Y}{Z}.
\end{align}
Here, the optimization is over all $n$-by-$n$ positive semi-definite complex matrices with diagonal entries equal to 1, which forms a convex set. Compared to (\ref{eqn:MLE_def_2}), the formulation in (\ref{eqn:SDP_def})  relaxes the rank constraint $\text{rank}(Z)=1$ to  $Z\succeq 0$, thus making the feasible set convex.



 
 While the SDP offers computational convenience, its solution is not guaranteed to be $\hat{z}^{\text{MLE}} (\hat{z}^{\text{MLE}})^\H$. A crucial question is how $\hat{Z}^{\text{SDP}}$ differs from $\hat{z}^{\text{MLE}} (\hat{z}^{\text{MLE}})^\H$. If they coincide, the SDP relaxation is considered tight in the literature.  \cite{bandeira2017tightness} demonstrates that the SDP is tight when $\sigma \lesssim n^{1/4}$. This result is further refined by \cite{zhong2018near}, which shows that the SDP is tight when $\sigma \lesssim \sqrt{n / \log n}$. However, the behavior of the SDP when $\sigma \gtrsim \sqrt{n / \log n}$, referred to as the high-noise regime in this paper, remains unclear. This motivates us to address the following question:

~\\
{\centering \emph{Question 1: How does the SDP differ from the MLE in the high-noise regime where $\sigma \gtrsim \sqrt{n/\log n}$?}}

~\\
\indent In addition to the SDP relaxation, in recent years, the Burer-Monteiro (BM) factorization \cite{burer2003nonlinear, burer2005local, boumal2016non, bandeira2016low, mcrae2024benign, ling2023local} has drawn increasing attention.  For any $m \in \mathbb{N}$, the BM factorization solves the following optimization problem:
\begin{align}\label{eqn:BM_m_def}
\hat Z^{\BM,m} =   \argmax_{Z\in\mathc^{n\times n}:,Z=Z^\H,\text{rank}(Z) \leq m,Z\succeq 0,Z_{jj}=1,\forall j\in[n]} \iprod{Y}{Z}.
\end{align}
Compared to the SDP, the BM factorization imposes an additional rank constraint. Note that when $m=1$, the feasible set of the BM factorization is the set of all rank-1 Hermitian matrices with the non-zero eigenvalue being 1. Hence, the BM factorization is equivalent the MLE in the sense that $$\hat Z^{\BM,1} = \hat z^\MLE (\hat z^\MLE)^\H.$$ When $m\geq n$, the rank constrain is not effective, and the BM factorization is equivalent to the SDP. 
As a result, the BM factorization can be seen as a more conservative relaxation of the MLE compared to the SDP when $m<n$. In addition, the SDP can be seen as a special case of the BM factorization, such that
\begin{align*}
\hat Z^{\BM, n} = \hat Z^{\SDP}.
\end{align*}
With the SDP seen as a special case of the BM factorization,  the question posed above about the SDP can be further generalized:

~\\
{\centering \emph{Question 2: How does the BM factorization differ from the MLE?}}
~\\

Note that the difference between the BM factorization and the MLE can be quantified by the following normalized squared Frobenius norm: $n^{-2}\|\hat Z^{\BM,m} - \hat z^{\MLE}(\hat z^{\MLE})^\H\|_{\rm F}^2$,  between $\hat Z^{\BM,m}$ and $\hat z^{\MLE}(\hat z^{\MLE})^\H$. Since both are $n\times n$ matrices with entries having modulus at most 1, the quantity is between 0 and 4. Hence, to address Question 2, we aim to establish an upper bound for $n^{-2}\|\hat Z^{\BM,m} - \hat z^{\MLE}(\hat z^{\MLE})^\H\|_{\rm F}^2$.  If it is equal to 0, then  $\hat Z^{\BM,m}=\hat z^{\MLE}(\hat z^{\MLE})^\H$ in which case we can say the BM factorization is tight.

The main results of this paper are presented below in Theorem \ref{thm:intro2}.

\begin{theorem}\label{thm:intro2}
Suppose $m\in\mathbb{N}\setminus\{1\}$.
\begin{enumerate}
\item There exists some absolute constant $C>0$ such that the following hold:
\begin{align}\label{eqn:intro2_1}
\E \br{\frac{1}{n^2}\fnorm{\hat Z^{\BM,m} - \hat z^{\MLE} \br{\hat z^{\MLE}}^\H}^2}& \leq C\ebr{-\frac{n}{8\sigma^2}} + 2n^{-10}.
\end{align}
\item There exists some absolute constant $C'>0$ such that if $\sigma\leq \min\{C'\sqrt{n}, \sqrt{n/(9\log n)}\}$, then the following holds  with high probability:
\begin{align*}
\hat Z^{\BM,m} = \hat z^{\MLE} \br{\hat z^{\MLE}}^\H. 
\end{align*}
\end{enumerate}
\end{theorem}

Theorem \ref{thm:intro2} first provides an upper bound (\ref{eqn:intro2_1}) for  the expected value of $n^{-2}\|\hat Z^{\BM,m} - \hat z^{\MLE}(\hat z^{\MLE})^\H\|_{\rm F}^2$. Note that $n^{-2}\|\hat Z^{\BM,m} - \hat z^{\MLE}(\hat z^{\MLE})^\H\|_{\rm F}^2$ is  random  as both $\hat Z^{\BM,m}$ and  $\hat z^\MLE$ depend on the random Gaussian noises $\{W_{jk}\}_{1\leq j<k\leq n}$. Therefore, we take the expectation to obtain a deterministic upper bound. The upper bound comprises two terms. The first term has an exponential form with $\frac{n}{\sigma^2}$ in the exponent, which can be understood as the signal-to-noise ratio. 
The second term $n^{-10}$ arises from a high-probability event controlling $\|W\|$ and can be made arbitrarily smaller, thus considered negligible compared to the first term. Ignoring the second term, the bound indicates the difference between $\hat Z^{\BM,m} $ and $\hat z^{\MLE}(\hat z^{\MLE})^\H$ is exponentially small.

To better understand the exponential error term $\ebr{-\frac{n}{8\sigma^2}}$ in (\ref{eqn:intro2_1}), particularly its magnitude, we compare it with the distances to the ground truth $z^*(z^*)^\H$. When $m=n$, the BM factorization becomes the SDP. \cite{gao2022sdp} shows that $n^{-2}\|\hat Z^{\SDP} -  z^{*}( z^{*})^\H\|_{\rm F}^2$, the difference between the SDP and the ground truth, is of the order $\frac{\sigma^2}{n}$. A similar result is established in \cite{gao2021exact} for the MLE. In addition, \cite{gao2021exact} demonstrates that $\frac{\sigma^2}{n}$ is the minimax rate for the estimation of $z^*(z^*)^\H$ in the phase synchronization, implying that no estimator can achieve an error much smaller than $\frac{\sigma^2}{n}$. The left panel of Figure \ref{fig:intro} visualizes the geometric relationship among these quantities. Note that $\ebr{-\frac{n}{8\sigma^2}}$ is much smaller than $\frac{\sigma^2}{n}$, especially when $\frac{n}{\sigma^2}$ is large. This reveals that while $\hat Z^\SDP$ and $\hat z^\MLE(\hat z^\MLE)^\H$ are distant from $z^*(z^*)^\H$, they are relatively close to each other, indication that relaxing  the feasible set in (\ref{eqn:MLE_def_2}) to that in (\ref{eqn:BM_m_def}) only slightly alters the solution.


\begin{figure}[h!]
    \centering
    \begin{minipage}{.45\textwidth}
        \centering
        \includegraphics[width=\textwidth]{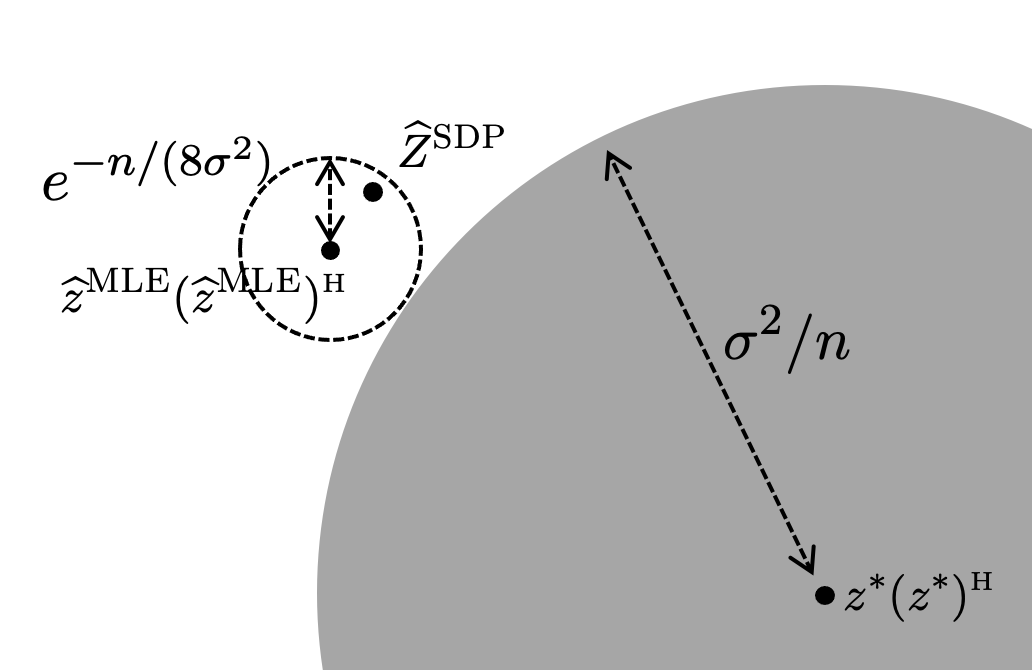}
    \end{minipage}\quad
    \begin{minipage}{0.5\textwidth}
        \centering
        \includegraphics[width=\textwidth]{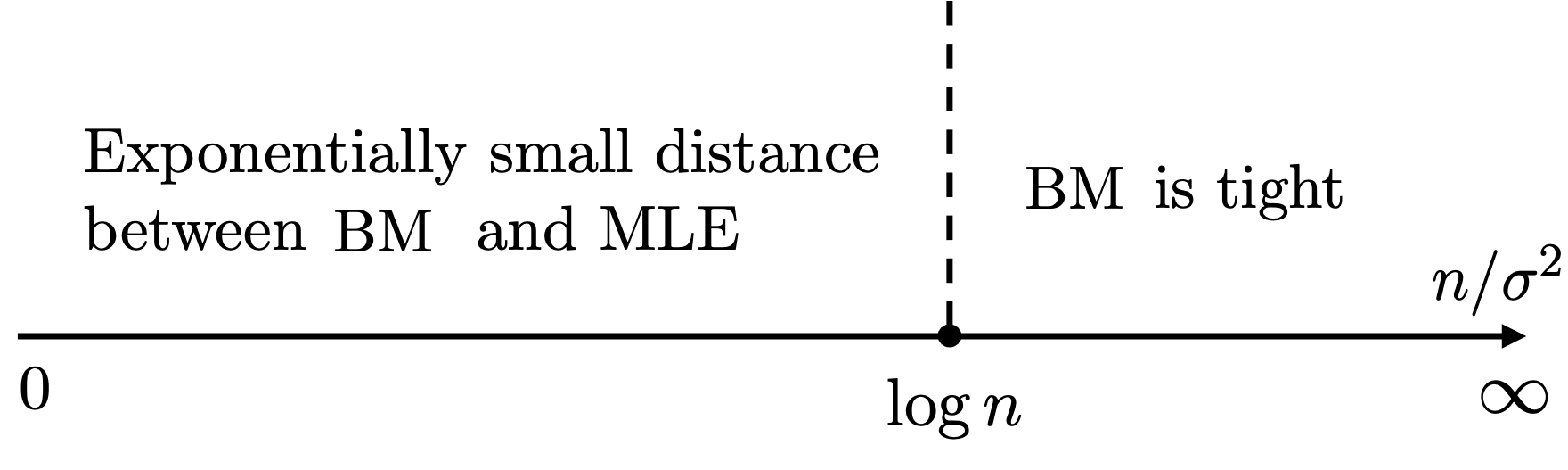}
    \end{minipage}
    \caption{Left: A visualization of the geometric relationship among the SDP, the MLE, and the ground truth. Right: Summary of Theorem \ref{thm:intro2}: The distance between the BM factorization and the MLE decays exponentially as $\frac{n}{\sigma^2}$ increases. The distance becomes 0, indicating tightness, when $\frac{n}{\sigma^2}\gtrsim \log n$.
     \label{fig:intro}}
\end{figure}

While (\ref{eqn:intro2_1}) holds for any noise level $\sigma$, Theorem \ref{thm:intro2} further provides a high-probability result for the tightness of the BM factorization. When $n$ is large enough, a sufficient condition is $\sigma\leq \sqrt{n/(9\log n)}$. Compared to the condition $\sigma \lesssim \sqrt{n/\log n}$ for the SDP in the existing literature \cite{zhong2018near}, our result  holds for any BM factorization, including the SDP.


In summary, Theorem \ref{thm:intro2} offers comprehensive answers to the two questions posed above. Theorem \ref{thm:intro2} directly addresses Question 2, demonstrating how the BM factorization differs from the MLE as $\frac{n}{\sigma^2}$ increases (see the right panel of Figure \ref{fig:intro} for an illustration). For Question 1, since the SDP is a special case of the BM factorization, our results indicate that the SDP differs from the MLE with an exponentially small error in the high-noise regime $\sigma \gtrsim \sqrt{n / \log n}$.


To establish Theorem \ref{thm:intro2}, our analysis is beyond that of \cite{zhong2018near}. \cite{zhong2018near} connects the SDP and the MLE through a dual certificate. Instead, we leverage a property that both the BM factorization and the MLE are fixed points of certain mappings $F_m$ and $F_1$, respectively (see Section \ref{sec:fixed_points} for their definitions). The fixed-point property applies not just to the SDP but also the BM factorization, allowing us to establish a general framework for the BM factorization that includes the SDP as a special case. More importantly, this enables us to first establish contraction-type results for these mappings, ultimately showing that the difference between the BM factorization and the MLE can be upper bounded by a quantity of the MLE (see Corollary \ref{cor:fixed_points_MLE} for more details). As a result, the remaining proof is about analyzing the MLE. While \cite{zhong2018near} investigates a similar quantity when $\sigma \lesssim \sqrt{n/\log n}$ by developing an $\ell_\infty$ norm analysis for the MLE, their approach no longer works in the high-noise regime $\sigma \gtrsim \sqrt{n/\log n}$, due to the fact that the mapping $F_1$ involves entrywise normalization, which poses analytical challenges.  To address this, our strategy is to replace $F_1$ with a Lipschitz mapping whose fixed points closely approximate $\hat z^\MLE$ and are relatively easier to analyze. This novel strategy, together with the leave-one-out technique developed in \cite{zhong2018near}, allows us to establish the desired exponential bound presented in Theorem \ref{thm:intro2}.

~\\
\emph{Related Literature.} Recent years have seen a surge of interest in SDP relaxations for tackling various non-convex optimization problems with underlying low-rank structures. This includes applications in community detection \cite{hajek2016achieving, fei2020achieving}, clustering \cite{giraud2019partial, fei2022hidden}, matrix completion \cite{candes2012exact}, and phase retrieval \cite{waldspurger2015phase}. Among these, phase synchronization is particularly notable for its relatively straightforward structure, making it a prime candidate for in-depth study.

The SDP in phase synchronization can be studied from several perspectives. From a statistical standpoint, as explored in \cite{javanmard2016phase, gao2022sdp}, the focus is on using SDP to estimate the true structure $z^*(z^*)^\H$ and to evaluate its estimation error and statistical optimality. This involves analyzing the distance between $\hat{Z}^\SDP$ and the true $z^*(z^*)^\H$. Conversely, studies such as \cite{mei2017solving, bandeira2017tightness, zhong2018near} concentrate on the tightness of the SDP—specifically, the discrepancy between $\hat{Z}^\SDP$ and $\hat{z}^\MLE(\hat{z}^\MLE)^\H$, essentially quantifying the cost of transforming a non-convex optimization problem into a convex one through SDP relaxation. \cite{mei2017solving} examines the differences between the objective function values, $\langle Y, \hat{Z}^\SDP \rangle$ and $\langle Y, \hat{z}^\MLE(\hat{z}^\MLE)^\H \rangle$, while \cite{bandeira2017tightness, zhong2018near} demonstrates that $\hat{Z}^\SDP$ equals $\hat{z}^\MLE(\hat{z}^\MLE)^\H$ under low noise conditions. These works have inspired further investigations, such as those by \cite{ling2022improved, ling2021near}, into the tightness of SDP in related problems, like orthogonal group synchronization and the generalized orthogonal Procrustes problem.
Our work extends the line of research initiated by \cite{zhong2018near} by focusing on the tightness of the SDP in the high-noise regime, an area not extensively covered by existing research. By doing so, we contribute to filling a crucial gap in understanding the limits and power of SDP relaxations under more challenging conditions. While our results are limited to the phase synchronization, they could  potentially be extended to other synchronization problems such as the orthogonal group synchronization.

While SDP is a convex optimization approach solvable in polynomial time, its scalability issues have prompted the exploration of alternatives such as the BM factorization \cite{burer2003nonlinear}. Despite its non-convex nature, BM factorization often exhibits surprisingly good performance when applied through local optimization algorithms. This observation has spurred a series of investigations into the conditions under which SDP and BM factorization yield equivalent optima \cite{burer2005local, boumal2016non, boumal2019global}. Furthermore, studies such as \cite{mcrae2024benign, ling2023local, mcrae2024nonconvex, endor2024benign} examine the landscape and benignness of the BM factorization's optimization process.
Unlike much of the existing literature that focuses on comparing BM factorization directly with SDP, our work considers both SDP and BM factorization as relaxations of the MLE. Therefore, we explore the differences between BM factorization and MLE, rather than between BM factorization and SDP.

Regarding the statistical properties of phase synchronization, \cite{gao2021exact} establishes the minimax rate of $\frac{\sigma^2}{n}$, demonstrating its attainability via the MLE and a generalized power method. Subsequent research by \cite{gao2022sdp, zhang2024exact} confirms that both the SDP and the eigenvector method \cite{singer2011angular} are minimax optimal. Phase synchronization serves as a specific instance within the broader framework of group synchronization problems \cite{abbe2017group}, where the elements $\{z^*_j\}_{j\in[n]}$ belong to various groups. The performance of several algorithms, including those mentioned, has also been studied in different synchronization settings such as $\mathbb{Z}_2$ synchronization \cite{gao2022sdp} and orthogonal group synchronization \cite{ling2022improved, zhang2024exact}.



~\\
\emph{Organization.} In Section \ref{sec:2}, we conduct a deterministic analysis of the difference between the MLE and the BM factorization using their fixed-point properties. Section \ref{sec:3} focuses on analyzing the MLE. 
We include proofs of main results in Section \ref{sec:4}.
Due to the page limit, proofs of remaining lemmas are included in the supplementary material.


~\\
\emph{Notations.} Define $\mathbb{N}=\{1,2,3,\ldots\}$ as the set of natural numbers. For any positive integer $n$, we write $[n]=\{1,2,\ldots, n\}$ and denote $I_n$ as the $n\times n$ identity matrix. For a complex number $x\in\mathc$, we use $\bar x$ for its complex conjugate, $\re(x)$ for its real part, and $|x|$ for its modulus.  
Define $\mathcleq = \{x\in\mathc:|x|\leq 1\}$ as the set of complex numbers whose modulus are at most 1.  For a complex vector $x=(x_j)\in\mathbb{C}^d$, we denote $\|x\|=({\sum_{j=1}^d |x_j|^2})^{1/2}$ as its Euclidean norm. 
 For a complex matrix $B =(B_{jk})\in\mathbb{C}^{d_1\times d_2}$, we  use $B^\H \in\mathbb{C}^{d_2\times d_1}$ for its conjugate transpose such that $(B^{\H})_{jk}= \overline{B_{kj}}$.  The Frobenius norm and the operator norm of $B$ are defined by $\fnorm{B}:=({\sum_{j=1}^{d_1}\sum_{k=1}^{d_2}|B_{jk}|^2})^{1/2}$ and $\norm{B} := \sup_{u\in\mathbb{C}^{d_1},v\in\mathbb{C}^{d_2}:\norm{u}=\norm{v}=1} u^\H Bv$. We use the notation  $B\succeq 0$ when $B$ is positive semi-definite. Define $B_j$ as its $j$th column and $B_{j\cdot}$ as its $j$th row. For a square matrix $B$, define $\Tr(B)$ as its trace and $\text{rank}(B)$ as its rank. For two matrices $A=(A_{jk})\in\mathc^{d_1\times d_2}$ and $B=(B_{jk})\in\mathc^{d_1\times d_2}$, define $\iprod{A}{B} = \Tr(A^\H B) =\sum_{j=1}^{d_1}\sum_{k=1}^{d_2} \overline{A_{jk}} B_{jk}$ as its Frobenius inner product.
 For two positive sequences $\{a_n\}$ and $\{b_n\}$, $a_n\lesssim b_n$ and $b_n\gtrsim a_n$ both mean $a_n\leq Cb_n$ for some constant $C>0$ independent of $n$. We also write $a_n=o(b_n)$ or $\frac{b_n}{a_n}\rightarrow\infty$ when $\limsup_n\frac{a_n}{b_n}=0$.  
We use $\indic{\cdot}$ as the indicator function. 
 

\section{A Deterministic Analysis Through Fixed Points}\label{sec:2}
In this section, we show the difference between the MLE and the BM factorization can be upper bounded by a quantity related to the MLE. Our analysis is deterministic, using a fact the estimators are fixed points.

\subsection{Introducing an Equivalent Representation of the BM Factorization and a Different Loss Function}
We first introduction an equivalent representation of the BM factorization. For any $m\in\mathbb{N}\setminus\{1\}$, define
\begin{align*}
\mathcal{V}_m=\{V=(V_1,\ldots, V_n)\in\mathc^{m\times n}: \norm{V_j}=1,\forall j\in[n]\}
\end{align*}
as a set containing all $m\times n$ complex matrices with unit norm columns. Note that for any $Z$ that is in the feasible set of (\ref{eqn:BM_m_def}), it can be represented as $Z = V^\H V$ for some $V\in\mathcal{V}_m$, and vice versa. Consequently, the BM factorization can be equivalently formulated as
\begin{align}\label{eqn:V_BM_m_def}
\hat V^{\BM,m} = \argmax_{V\in\mathcal{V}_m} \iprod{Y}{V^\H V},
\end{align}
with $\hat Z^{\BM,m} = (\hat V^{\BM,m})^\H \hat V^{\BM,m}$. Since MLE and SDP are special cases of the BM factorization, we have $\hat z^\MLE = (\hat V^{\BM,1})^\H$ and $\hat Z^\SDP = (\hat V^{\SDP})^\H \hat V^{\SDP}$, where we define $\hat V^{\SDP} = \hat V^{\BM,n}$. As we will show in Section \ref{sec:fixed_points}, $\hat V^{\BM,m}$ is a fixed point of certain mapping, a critical property on which our analysis is built.

While the difference between the BM factorization and the MLE can be captured by  the normalized squared Frobenius norm $n^{-2}\|\hat Z^{\BM,m} - \hat z^\MLE (\hat z^\MLE)^\H\|_{\rm F}^2$, we opt to quantify it through the deviation between $\hat V^{\BM,m}$ and $\hat z^\MLE$.
Consider the following loss function $\ell_m:\mathc^{m\times n}\times \mathc^n\rightarrow \mathr$ defined as
\begin{align}\label{eqn:lm_def}
\ell_m(V,z)  = \min_{a\in\mathc^m:\norm{a}=1} \frac{1}{n} \fnorm{V - a z^\H}^2,
\end{align}
for any $V\in \mathc^{m\times n}$ and any $z\in \mathc^n$.
Then the deviation can be measured by $\ell_m(\hat V^{\BM,m}, \hat z^\MLE)$. 
The advantage of studying $\ell_m(\hat V^{\BM,m}, \hat z^\MLE)$ instead of $n^{-2}\|\hat Z^{\BM,m} - \hat z^\MLE (\hat z^\MLE)^\H\|_{\rm F}^2$ is twofold. First,
the two quantities are closely related  through the following inequality (see Lemma \ref{lem:loss} for its proof):
\begin{align}\label{eqn:loss_connection}
\frac{1}{n^2}\fnorm{\hat Z^{\BM,m} - \hat z^\MLE (\hat z^\MLE)^\H}^2 \leq 2\ell_m(\hat V^{\BM,m},\hat z^\MLE).
\end{align}
Hence, in order to establish Theorem \ref{thm:intro2}, it is sufficient to upper bound $\ell_m(\hat V^{\BM,m},\hat z^\MLE)$. Second and more importantly,
in our analysis, we view $\hat V^{\BM,m}$ and $\hat z^\MLE$ as fixed points of certain mappings, a property that it is more natural to exploit with $\ell_m(\hat V^{\BM,m}, \hat z^\MLE)$. As a result, in the remaining part of the paper, we will focus on analyzing and upper bounding $\ell_m(\hat V^{\BM,m}, \hat z^\MLE)$.

\subsection{Introducing $F_1$ and $F_m$: $\hat z^\MLE$ and $\hat V^{\BM,m}$ Are Their Fixed Points}\label{sec:fixed_points}
Our analysis relies on the fact that the MLE and the BM factorization are both fixed points.
Define a function $F_1:\mathc_1^n\rightarrow\mathc_1^n$ such that for any $z\in\mathc_1^n$, the $j$th coordinate of $F_1(z)$ is
\begin{align*}
[F_1(z)]_j = \begin{cases}
\frac{\sum_{k\in[n] } Y_{jk}z_k}{|\sum_{k\in[n] } Y_{jk}z_k|},\text{ if }\sum_{k\in[n] } Y_{jk}z_k\neq 0,\\
z_j, \text{ o.w.,}
\end{cases}\quad \forall j\in[n].
\end{align*}
It can be written equivalently as
\begin{align}\label{eqn:F_1_def}
[F_1(z)]_j = \begin{cases} 
\frac{[Y z]_j}{|[Yz]_j|},\text{ if }[Yz]_j\neq 0,\\
z_j,\text{ o.w.,}
\end{cases}\quad \forall j\in[n].
\end{align}
Then the MLE is a fixed point of the mapping $F_1$, that is,
\begin{align*}
\hat z^\MLE = F_1\br{\hat z^\MLE}.
\end{align*}
To see this, recall from the definition in (\ref{eqn:MLE_def}) that \(\hat{z}^{\MLE}\) maximizes the objective \(\langle Y, z z^\H \rangle\) over all \(z \in \mathbb{C}_1^n\). Fix any index \(j \in [n]\). We can write the objective at $\hat z^\MLE$ as
\[
\langle Y, \hat{z}^{\MLE} (\hat{z}^{\MLE})^\H \rangle = 2 \re\Bigg( \overline{\hat{z}^{\MLE}_j} \sum_{k\in[n]} Y_{jk} \hat{z}^{\MLE}_k \Bigg) + \text{(terms independent of } \hat{z}^{\MLE}_j).
\]
Since \(\hat{z}^{\MLE}\) is the global maximizer, its \(j\)th coordinate must maximize the first term of the above expression over $\mathc_1$. This implies
\[
\hat{z}^{\MLE}_j = \argmax_{z_j \in \mathbb{C}_1} \re\Bigg( \overline{z_j} \sum_{k\in[n]} Y_{jk} \hat{z}^{\MLE}_k \Bigg) = \frac{\sum_{k\in[n]}  Y_{jk} \hat{z}^{\MLE}_k}{\left| \sum_{k\in[n]}  Y_{jk} \hat{z}^{\MLE}_k \right|} = [F_1(\hat{z}^{\MLE})]_j,
\]
as long as $\sum_{k\in[n]}  Y_{jk} \hat{z}^{\MLE}_k\neq 0$.
Therefore, every coordinate of \(\hat{z}^{\MLE}\) satisfies the fixed-point condition, so \(\hat{z}^{\MLE} = F_1(\hat{z}^{\MLE})\).

In addition, for any $m\in\mathbb{N}\setminus\{1\}$, define a function $F_m: \mathcal{V}_m\rightarrow \mathcal{V}_m$ such that for any $V\in \mathcal{V}_m$,  the $j$th column of $F_m(V)$ is
\begin{align*}
[F_m(V)]_j = \begin{cases}
\frac{\sum_{k\in[n] }\bar{Y}_{jk} V_k}{\norm{\sum_{k\in[n] } \bar{Y}_{jk} V_k}},\text{ if }\sum_{k\in[n] } \bar{Y}_{jk} V_k\neq 0,\\
V_j,\text{ o.w.,}
\end{cases}\quad \forall j\in[n].
\end{align*}
It can be written equivalently as
\begin{align*}
[F_m(V)]_j = \begin{cases}
\frac{[V Y^\H]_j}{\norm{[V Y^\H]_j}},\text{ if }[V Y^\H]_j\neq 0,\\
V_j,\text{ o.w.,}
\end{cases}\quad \forall j\in[n].
\end{align*}
From the definition in (\ref{eqn:V_BM_m_def}), $\hat V^{\BM,m}$ maximizes the objective $\iprod{Y}{V^\H V}$ over all $V\in\mathcal{V}_m$. Following an argument analogous to the one above for \(\hat{z}^{\MLE}\), it can be shown that $\hat V^{\BM,m}$ is fixed point of the mapping $F_m$, that is,
\begin{align*}
\hat V^{\BM,m} &= F_m\br{\hat V^{\BM,m}}.
\end{align*}


The fact that $\hat{z}^\MLE$ and $\hat{V}^{\BM,m}$ are fixed points of $F_1$ and $F_m$ opens a door for our analysis. In Section \ref{sec:contraction}, we will first establish a contraction-type result for the mappings. Consider any $z \in \mathc_1^n$ and $V \in \mathcal{V}_m$. The key is to understand how $\ell_m(F_m(V), F_1(z))$ depends on $\ell_m(V, z)$. We aim to establish the following contraction-type result:
\begin{align}\label{eqn:main_sec_1}
\ell_m(F_m(V), F_1(z)) \leq \text{(some factor smaller than 1)} \times \ell_m(V, z) + \text{ some additive error term}.
\end{align}
On a high level, it means the mappings jointly have a contraction but with some additive error term.  
If $V, z$ are fixed points such that $F_m(V) = V$ and $F_1(z) = z$, we have $\ell_m(F_m(V), F_1(z)) = \ell_m(V, z)$, and (\ref{eqn:main_sec_1}) consequently becomes 
\begin{align}\label{eqn:main_sec_2}
\ell_m(V, z) \leq \text{(some factor smaller than 1)} \times \ell_m(V, z) + \text{ some additive error term},
\end{align}
an inequality involving $\ell_m(V, z)$ on both sides. Then the term $\ell_m(V, z)$ on the right-hand side can be absorbed into the one on the left-hand side, leading to an upper bound for $\ell_m(V, z)$ with an explicit expression. 
Since $V^{\BM,m}$ and $z^\MLE$ are fixed points, the derived upper bound holds for $\ell_m(V^{\BM,m}, z^\MLE)$. In Section \ref{sec:implication}, we include the result for $\ell_m(V^{\BM,m}, z^\MLE)$ with discussions.

We would like to clarify that the use of \(\hat{z}^{\MLE}\) and \(\hat{V}^{\BM,m}\) as fixed points in this paper is purely analytical: we explicitly exploit their fixed-point properties to quantify the theoretical difference between them. One might be tempted to use the mappings \(F_1\) and \(F_m\) for computation—for example, by iteratively applying \(F_1\) starting from an initialization until convergence. This procedure is known as the generalized power method and has been studied in prior work \cite{boumal2016nonconvex, liu2017estimation, zhong2018near}. In particular, convergence of the method is established in the low-noise regime where \(\sigma \lesssim \sqrt{n/\log n}\) \cite{zhong2018near}.
However, due to the non-linear nature of the mappings \(F_1\) and \(F_m\), the behavior of these iterative algorithms in the high-noise regime remains less understood. In particular, convergence may not be guaranteed when \(\sigma\) is large, and the conditions required for initialization in such regimes are still unclear. Further discussion on computational aspects is provided in Section~\ref{sec:computation}.

\subsection{A Contraction-type Result for $\ell_m(F_m(V), F_1(z))$ and $\ell_m(V, z)$}\label{sec:contraction}

In this section, we aim to establish (\ref{eqn:main_sec_1}). To achieve this, we need to study the two mappings $F_1$ and $F_m$. Note that they can both be decomposed into two similar steps, as demonstrated below:
\begin{align*}
&z \; \xrightarrow{\hspace*{5cm}} \; Yz \;\; \xrightarrow{\hspace*{5cm}} F_1(z)\\
&\xdownarrow{0.6cm}\ell_m(V,z) \quad \quad \parbox[t]{0.1\textwidth}{
                    matrix\\
                    multiplication} \quad \quad \quad\quad\quad \xdownarrow{0.6cm} \ell_m(VY^\H,Yz) \quad\quad \parbox[t]{0.1\textwidth}{
                    entrywise\\
                    normalization}\quad  \quad\quad \xdownarrow{0.6cm}  \ell_m(F_m(V), F_1(z))\\
&V \xrightarrow{\hspace*{5cm}} VY^\H \xrightarrow{\hspace*{5cm}} F_m(V)
\end{align*}
In the first step, they involve a matrix multiplication with the data matrix $Y$ such that $z$ becomes $Yz$ and $V$ becomes $VY^\H$. In the second step, they perform an entrywise (or column-wise) normalization such that $Yz$ becomes $F_1(z)$ and $VY^\H$ becomes $F_m(V)$. Consequently, our analysis is decomposed into two parts.

For the first part (the matrix multiplication part), Lemma \ref{lem:lm_contraction} shows that $\ell_m(V Y^\H, Yz)$ can be upper bounded by $\ell_m(V, z)$ up to some factor, provided that $V$ and $z$ are close to the ground truth $z^*$. The closeness of $V$ to $z^*$ can be measured by $\ell_m(V, z^*)$. Regarding $z$, we define a loss function in an analogous way. Define a loss function $\ell_1: \mathc_1^n \times \mathc_1^n \rightarrow \mathr$ such that for any $z, z' \in \mathc^n$,
\begin{align}\label{eqn:ell_1_def}
\ell_1(z', z) = \min_{a \in \mathc_1} \frac{1}{n} \|z' - az\|^2 = 2 - n^{-1} |(z')^\H z|.
\end{align}
Then, the closeness of $z$ to $z^*$ can be measured by $\ell_1(z, z^*)$. Lemma \ref{lem:lm_contraction} is an extension of Lemma 12 of \cite{zhong2018near}, which proves the vector case that connects $\ell_1(Yz', Yz)$ with $\ell_1(z', z)$ for $z, z' \in \mathc_1^n$ that are close to $z^*$. We generalize it to the matrix case.
\begin{lemma}\label{lem:lm_contraction}
Suppose $m \in \mathbb{N} \setminus \{1\}$ and $\epsilon \in (0, 1/2)$. For any $z \in \mathc_1^n$ such that $\ell_1(z, z^*) \leq \epsilon^2$ and any $V \in \mathcal{V}_m$ such that $\ell_m(V, z^*) \leq \epsilon^2$, we have
\begin{align*}
\ell_m(V Y^\H, Yz) \leq n^2 \left(6\epsilon + \frac{\sigma \norm{W}}{n}\right)^2 \ell_m(V, z).
\end{align*}
\end{lemma}


For the second part (the normalization part), we need to study how the normalization affects the loss function $\ell_m$ to connect $\ell_m(V Y^\H, Yz)$ with $\ell_m(F_m(V), F_1(z))$. However, this is not straightforward as the normalization operation is not continuous and, more importantly, does not necessarily have a contraction property. To see this, consider any $t > 0$ and any $x, y \in \mathc$ such that $|x| = |y| = t$. Then $\left| \frac{x}{|x|} - \frac{y}{|y|} \right| = t^{-1} |x - y|$. If $t > 1$, then $x$ and $y$ get closer after normalization. If $t < 1$ and is close to 0, the distance between $\frac{x}{|x|}$ and $\frac{y}{|y|}$ can be much larger than that between $x$ and $y$, though the distance is capped at most 2. The following Lemma \ref{lem:normalization} shows how normalization changes the distance between vectors. In the lemma, (\ref{eqn:lem_normalization_1}) is a simple case where two vectors $x, y$ are non-zero; (\ref{eqn:lem_normalization_2}) allows $x, y$ to be zero, which can be used to analyze the normalizations in $F_1$ and $F_m$.

\begin{lemma}\label{lem:normalization}
Suppose $m\in\mathbb{N}$. For any vectors $x,y\in\mathc^m\setminus\{0\}$ and for any $t>0$, we have
\begin{align}\label{eqn:lem_normalization_1}
\norm{\frac{x}{\norm{x}} - \frac{y}{\norm{y}}} \leq  \frac{2\norm{x-y}}{t} + 2\indic{\norm{y}<  t}.
\end{align}
For any vectors $x,y,u,v\in\mathc^m$ such that $\norm{u},\norm{v}\leq 1$ and for any $t>0$, we have
\begin{align}\label{eqn:lem_normalization_2}
\norm{\br{\frac{x}{\norm{x}}\indic{x\neq 0} + u\indic{x=0}} - \br{ \frac{y}{\norm{y}}\indic{y\neq 0} + v\indic{y=0}}  } \leq  \frac{2\norm{x-y}}{t} + 2\indic{\norm{y}<  t}.
\end{align}
\end{lemma}


Lemma \ref{lem:normalization} (more specifically, (\ref{eqn:lem_normalization_2})) can be applied to analyze the difference between each column-wise normalization of $VY^\H$ and each coordinate-wise normalization of $Yz$ for any threshold $t > 0$. To be more specific, for the $j$th normalization of $VY^\H$ and $Yz$, the application of (\ref{eqn:lem_normalization_2}) leads to two terms, corresponding to the two terms in the upper bound of (\ref{eqn:lem_normalization_2}): the first term is essentially about the distance between $[VY^\H]_j$ and $[Yz]_j$, and the second term is about $\indic{|[Yz]_j| < t}$. Aggregated over all $j \in [n]$, the first term can be related to $\ell_m(VY^\H, Yz)$, which can be further bounded by Lemma \ref{lem:lm_contraction}; the second term becomes $\sum_{j \in [n]} \indic{|[Yz]_j| < t}$.
This leads to the following theorem, which gives a connection between $\ell_m(F_m(V), F_1(z))$ and $\ell_m(V, z)$.

\begin{theorem}\label{thm:fixed_points}
Suppose $m\in\mathbb{N}\setminus\{1\}$ and  $\epsilon\in(0,1/2)$. For any $z\in\mathc_1^n$ such that $\ell_1(z,z^*)\leq \epsilon^2$ and any $V\in\mathcal{V}_m$ such that $\ell_m(V,z^*)\leq \epsilon^2$, we have
\begin{align}\label{eqn:main_sec_3}
\ell_m(F_m(V),F_1(z)) \leq  \frac{4n^2}{t^2}\br{6\epsilon + \frac{\sigma\norm{W}}{n}}^2 \ell_m(V,z) + \frac{4}{n}  \sum_{j\in[n]}\indic{ \abs{[Yz]_j} <t},\quad \forall t>0.
\end{align}
If $z,V$ are further assumed to satisfy $z=F_1(z)$ and $V=F_m(V)$, we have
\begin{align}\label{eqn:main_sec_4}
 \ell_m(V,z)\leq \frac{8}{n}  \sum_{j\in[n]}\indic{ \abs{[Yz]_j} <\delta n},\quad \forall \delta \geq 2\sqrt{2}\br{6\epsilon + \frac{\sigma\norm{W}}{n}}.
\end{align}
\end{theorem}

%

In Theorem \ref{thm:fixed_points}, (\ref{eqn:main_sec_3}) holds for any threshold $t > 0$ and any $z \in \mathc_1^n$ and $V \in \mathcal{V}_m$ that are close to $z^*$. With a sufficiently large choice of $t$, the factor $\frac{4n^2}{t^2}\left(6\epsilon + \frac{\sigma \norm{W}}{n}\right)^2$ is smaller than 1, leading to the establishment of (\ref{eqn:main_sec_1}). Under this scenario, $F_m$ and $F_1$ jointly have a contraction-type property: after one iteration, $F_m(V)$ and $F_1(z)$ get closer, up to an additive error $\frac{4}{n} \sum_{j \in [n]} \indic{\abs{[Yz]_j} < t}$, compared to $V$ and $z$, with respect to the loss $\ell_m$.

In Theorem \ref{thm:fixed_points}, (\ref{eqn:main_sec_4}) is an immediate consequence of (\ref{eqn:main_sec_3}) if $z$ and $V$ are further assumed to be fixed points, following the argument as in (\ref{eqn:main_sec_2}). (\ref{eqn:main_sec_4}) shows that the distance between $V$, a fixed point of $F_m$, and $z$, a fixed point of $F_1$, provided that they are close to $z^*$, can be upper bounded by $\frac{8}{n} \sum_{j \in [n]} \indic{\abs{[Yz]_j} < \delta n}$, a property of $z$. This property essentially concerns the number of coordinates in $Yz$ whose absolute values are smaller than a certain threshold. If there is no such coordinate, i.e., $\sum_{j \in [n]} \indic{\abs{[Yz]_j} < \delta n} = 0$, then (\ref{eqn:main_sec_4}) leads to $\ell_m(V, z) = 0$.

\subsection{Implications on $ \ell_m(\hat V^{\BM,m},\hat z^\MLE)$}\label{sec:implication}

Since $\hat V^{\BM,m}$ and $\hat z^\MLE$ are fixed points of $F_m$ and $F_1$, respectively, a direct consequence of (\ref{eqn:main_sec_4}) in Theorem \ref{thm:fixed_points} is the following corollary  for $ \ell_m(\hat V^{\BM,m},\hat z^\MLE)$.

\begin{corollary}\label{cor:fixed_points_MLE}
Suppose $m\in\mathbb{N}\setminus\{1\}$ and $\epsilon\in(0,1/2)$. If $ \ell_1(\hat z^\MLE,z^*)\leq \epsilon^2$ and $\ell_m(\hat V^{\BM,m},z^*)\leq \epsilon^2$ are satisfied, we have
\begin{align*}
 \ell_m(\hat V^{\BM,m},\hat z^\MLE)\leq \frac{8}{n}  \sum_{j\in[n]}\indic{\abs{[Y\hat z^\MLE]_j}<\delta n},\quad \forall \delta \geq 2\sqrt{2}\br{6\epsilon + \frac{\sigma\norm{W}}{n}}.
\end{align*}
\end{corollary}


Corollary \ref{cor:fixed_points_MLE} reveals that the distance between $\hat{V}^{\BM,m}$ and $\hat{z}^\MLE$ can be upper bounded by a quantity of the MLE: $\frac{8}{n} \sum_{j \in [n]} \indic{\abs{[Y \hat{z}^\MLE]_j} < \delta n}$, which is essentially the proportion of coordinates in the vector $Y \hat{z}^\MLE$ whose absolute value is smaller than $\delta n$. Below are some remarks about this quantity and Corollary \ref{cor:fixed_points_MLE}.

\emph{1) Connection between the tightness of the BM factorization and the $\ell_\infty$ norm analysis.} It turns out an $\ell_\infty$ norm analysis for $W \hat z^\MLE$ is sufficient to show $ \ell_m(\hat V^{\BM,m},\hat z^\MLE)=0$, i.e., the tightness of the BM factorization. To see this, recall the decomposition of $Y$ in (\ref{eqn:phase_model_matrix}). Then for each $j\in[n]$,
\begin{align*}
[Y\hat z^\MLE]_j &= [(z^*(z^*)^\H + \sigma W) \hat z^\MLE]_j = z^*_j ((z^*)^\H \hat z^\MLE) + \sigma [W \hat z^\MLE]_j.
\end{align*}
Since $|(z^*)^\H \hat z^\MLE| = n(1-\ell_1(\hat z^\MLE, z^*)/2)\geq n(1-\epsilon^2/2)$ according to (\ref{eqn:ell_1_def}), we have
\begin{align}\label{eqn:new_1}
|[Y\hat z^\MLE]_j | \geq n(1-\epsilon^2/2) - \sigma |  [W \hat z^\MLE]_j| \geq n(1-\epsilon^2/2) - \sigma\norm{W \hat z^\MLE}_\infty.
\end{align}
Hence, if $\sigma\norm{W \hat z^\MLE}_\infty \leq n(1-\delta - \epsilon^2/2)$ holds, then $ \sum_{j\in[n]}\indic{\abs{[Y\hat z^\MLE]_j}<\delta n} =0$, and consequently $ \ell_m(\hat V^{\BM,m},\hat z^\MLE)=0$. As a result, it is sufficient to carry out  an $\ell_\infty$ norm analysis for the MLE to establish $ \ell_m(\hat V^{\BM,m},\hat z^\MLE)=0$, which is achievable when $\sigma$ is small enough, as shown in \cite{zhong2018near} for the SDP.

\emph{2) Corollary \ref{cor:fixed_points_MLE} is beyond the tightness of the BM factorization and consequently requires analysis beyond the existing $\ell_\infty$ norm framework.} Corollary \ref{cor:fixed_points_MLE} is not just about establishing $ \ell_m(\hat V^{\BM,m},\hat z^\MLE)=0$, the tightness of the BM factorization. In fact, it quantifies the deviation by $ \ell_m(\hat V^{\BM,m},\hat z^\MLE)$ and upper  bound it by $\frac{8}{n}\sum_{j\in[n]}\indic{\abs{[Y\hat z^\MLE]_j}<\delta n}$, the analysis of which   is actually beyond the $\ell_\infty$ norm of $W\hat z^\MLE $.  To see this, from (\ref{eqn:new_1}), we have that for each $j\in[n]$, 
\begin{align*}
\indic{\abs{[Y\hat z^\MLE]_j}<\delta n} &\leq \indic{n(1-\epsilon^2/2) -  \sigma |  [W \hat z^\MLE]_j| < \delta n}\\
& \leq \indic{\sigma |  [W \hat z^\MLE]_j| > n(1-\delta - \epsilon^2/2)}.
\end{align*}
Hence, the upper bound in Corollary \ref{cor:fixed_points_MLE} can be further upper bounded by
\begin{align}\label{eqn:main_sec_5}
\frac{8}{n}\sum_{j\in[n]}\indic{\abs{[Y\hat z^\MLE]_j}<\delta n}\leq \frac{8}{n}\sum_{j\in[n]} \indic{\sigma |  [W \hat z^\MLE]_j| > n(1-\delta - \epsilon^2/2)},
\end{align}
which is essentially about the proportion of coordinates in $W\hat z^\MLE$ that is larger than certain threshold in absolute value. To bound it, we need to study entrywise behavior of $W \hat z^\MLE$ instead of its $\ell_\infty$ norm.

\emph{3) Intuition on the exponential bound in Theorem \ref{thm:intro2} for $\frac{8}{n}\sum_{j\in[n]}\indic{\abs{[Y\hat z^\MLE]_j}<\delta n}$.} From Corollary \ref{cor:fixed_points_MLE}, it is clear that in order to establish our main result Theorem \ref{thm:intro2}, it is sufficient to analyze the MLE $\hat z^\MLE$ to provide an upper bound for the quantity $\frac{8}{n}\sum_{j\in[n]}\indic{\abs{[Y\hat z^\MLE]_j}<\delta n}$. However, directly establishing the exponential bound for it is not easy, due to the dependence between $\hat z^\MLE$ and $Y$ (through the noise matrix $W$). Nevertheless, here we assume they are independent from each other to provide some intuition why it has exponential upper bound. Recall (\ref{eqn:main_sec_5}) holds.
Then each coordinate of $W \hat z^\MLE$ follows a Gaussian distribution as $W$ is a Gaussian matrix. Though they are not completely independent from each other, we can show $\frac{1}{n}\sum_{j\in[n]} \indic{\sigma |  [W \hat z^\MLE]_j| > n(1-\delta - \epsilon^2/2)}$ concentrates on its expected values,  $\pbr{{\sigma |  [W \hat z^\MLE]_1| > n(1-\delta - \epsilon^2/2)}}$, which is a Gaussian  tail probability. With small $\delta,\epsilon$, it can be bounded explicitly by $\ebr{-\frac{cn}{\sigma^2}}$ for some constant $c>0$. This provides an intuition to explain why the bound in Theorem \ref{thm:intro2} takes this form. On the other hand, this is purely just an intuition as $\hat z^\MLE$ and $Y$ are actually highly dependent on each other. 


To conclude this section,  note that Corollary \ref{cor:fixed_points_MLE} requires $\hat z^\MLE$ and $\hat V^{\BM,m}$ to be close to the ground truth $z^*$. The following lemma shows that $\ell_m(\hat V^{\BM,m},z^*)$ and $\ell_1(\hat z^\MLE,z^*)$ are upper bounded by $\frac{8\sigma \norm{W}}{n}$. Hence, when $\sigma$ is small such that $\frac{8\sigma \norm{W}}{n}\leq \epsilon^2$, the assumptions needed in  Corollary \ref{cor:fixed_points_MLE} are satisfied, and then the conclusion established therein holds.
\begin{lemma}\label{lem:MLE_SDP_Loose_bound}
For any $m\in\mathbb{N}\setminus\{1\}$, we have
$\ell_m(\hat V^{\BM,m},z^*)\leq \frac{8\sigma \norm{W}}{n}.$
In addition, the same upper bound holds for $\ell_1(\hat z^\MLE,z^*)$.
\end{lemma}

\section{Analysis on MLE}\label{sec:3}


From Corollary \ref{cor:fixed_points_MLE}, it is evident that to establish our main result, Theorem \ref{thm:intro2}, it suffices to analyze the MLE $\hat{z}^\MLE$ and provide an upper bound for the quantity $\frac{1}{n}\sum_{j \in [n]}\indic{\abs{[Y \hat{z}^\MLE]_j} < \delta n}$. As demonstrated in Section \ref{sec:implication}, if $\hat{z}^\MLE$ is independent of $Y$, an exponential upper bound can be established immediately. However, the dependence between them complicates the problem, requiring a delicate analysis. In Section \ref{sec:high}, we provide an overview of our analysis, which can be decomposed into two steps detailed in Sections \ref{sec:step1} and \ref{sec:step2}. The main results are given in Section \ref{sec:exponential}.

\subsection{High-level Idea of Our Analysis}\label{sec:high}

In this section, we present the high-level idea of our analysis, which is quite technical and involved. It consists of the following two steps.

~\\ 
 \emph{Step 1: Approximate $\hat{z}^\MLE$ by fixed points of Lipschitz mappings.} Recall that $\hat{z}^\MLE$ is a fixed point of $F_1$. As discussed in Section \ref{sec:2}, the normalization in $F_1$ complicates the analysis. To address this, note that for any $x \in \mathc$, the normalization operation $x \rightarrow x/|x|$ can be approximated by a function $x \rightarrow x/\max\{|x|, t\}$ for some tuning parameter $t > 0$ (denoted as $g_t(\cdot)$ in (\ref{eqn:gt_def})).  

 The advantage of using this approximated mapping is two-fold. First, the approximation error can be controlled, as $|x/|x| - x/\max\{|x|, t\}| \leq \indic{|x| < t}$. Second and more importantly, the mapping is Lipschitz (see Lemma \ref{lem:gt}). With the help of this mapping, we define another Lipschitz mapping $G(\cdot,\cdot,\cdot)$ (see (\ref{eqn:G_def}) for its definition) whose fixed points are used to approximate $\hat z^\MLE$.
 Specifically, in Lemma \ref{lem:MLE_extreme_points_upper_bound_2}, we show that for a suitable $\delta$, we have
\begin{align*}
\frac{1}{n}\sum_{j \in [n]}\indic{\abs{[Y \hat{z}^\MLE]_j} < \delta n} \leq \frac{9}{n}\sum_{j \in [n]} \indic{\abs{z_j^* |(z^*)^\H \hat{z}^\MLE| + \sigma [W z]_j} < 2 \delta n},
\end{align*}
where $z$ is any fixed point of $G(\cdot, |(z^*)^\H \hat{z}^\MLE|, 2\delta n)$. In this way, we reduce the problem to a fixed point analysis for $G(\cdot, (z^*)^\H \hat{z}^\MLE, 2\delta n)$.

However, note that in $G(\cdot, |(z^*)^\H \hat{z}^\MLE|, 2\delta n)$, the quantity $|(z^*)^\H \hat{z}^\MLE|$ is still random as it involves $\hat{z}^\MLE$. To completely remove $\hat{z}^\MLE$ from the above expression, we approximate $|(z^*)^\H \hat{z}^\MLE|$ by a grid of scalars. In this way, the problem becomes: given some $s,t$, how to upper bound $\frac{1}{n}\sum_{j \in [n]} \indic{\abs{z_j^* s + \sigma \left[ W z \right]_j} < \text{some threshold}}$, or more conveniently
\begin{align}\label{eqn:main_sec_6}
\frac{1}{n}\sum_{j \in [n]} \indic{\sigma\abs{  \left[ W z \right]_j} > \text{some threshold}},
\end{align}
where $z$ is a fixed point of $G(\cdot, s, t)$.

~\\
 \emph{Step 2: Leave-one-out analysis for fixed points of $G(\cdot, s, t)$.} Upper bounding (\ref{eqn:main_sec_6}) is still challenging because the mapping $G(\cdot, s, t)$ involves the noise matrix $W$ for any given $s, t$. Consequently, $z$, a fixed point of $G(\cdot, s, t)$, also depends on $W$. Note that for each $j \in [n]$, $[Wz]_j = W_{j\cdot} z$. The key to decoupling this dependence is to approximate $z$ by some quantity $z^{(-j)}$ that is close but independent of $W_{j\cdot}$, such that $W_{j\cdot} z \approx W_{j\cdot} z^{(-j)}$, which follows a Gaussian distribution. As a result, (\ref{eqn:main_sec_6}) becomes
\begin{align}\label{eqn:main_sec_10}
\frac{1}{n}\sum_{j \in [n]} \indic{\sigma \abs{W_{j\cdot} z^{(-j)}} \geq \text{some threshold}},
\end{align}
which can be analyzed and leads to the desired exponential bound.

Now, the problem becomes finding the desired $z^{(-j)}$ for $z$. To achieve this, we use the idea of leave-one-out. Let $W^{(-j)}$ be a matrix equal to $W$ but with its $j$th row and column zeroed out. Let $G^{(-j)}$ be a function equal to $G$ but using $W^{(-j)}$ instead of $W$, and let $z^{(-j)}$ be its fixed point. By definition, $z^{(-j)}$ is independent of $W_{j\cdot}$. On the other hand, since $W$ and $W^{(-j)}$ only differ by one column and one row, intuitively, the two functions $G$ and $G^{(-j)}$ do not differ much, and their fixed points are consequently close.

To establish the closeness of $z$ and $z^{(-j)}$ rigorously, we use the following method. Let $z^{(0)} = z^*$. Starting from it, we apply $G(\cdot, s, t)$ iteratively to obtain a sequence of vectors $z^{(0)}, z^{(1)}, z^{(2)}, \ldots$. We can show that this sequence converges to a fixed point. So we let $z = z^{(\infty)}$. Instead of $G(\cdot, s, t)$, we can apply $G^{(-j)}(\cdot, s, t)$ iteratively, which leads to another sequence $z^{(0, -j)} = z^*, z^{(1, -j)}, z^{(2, -j)}, \ldots, z^{(-j)} = z^{(\infty, -j)}$. It is evident that $\|z^{(0, -j)} - z^{(0)}\| = 0$. Using mathematical induction, we can show (see Lemma \ref{lem:leave_one_out_close}) that
\begin{align}\label{eqn:main_sec_7}
\norm{z^{(T)} - z^{(T, -j)}} \leq 3, \forall T \in \mathbb{N}.
\end{align}
Hence, the same result holds for the limit, which is $\norm{z - z^{(-j)}}$.

%

~\\
\indent The details of these two steps are included in Sections \ref{sec:step1} and \ref{sec:step2}. Together, they lead to the establishment of exponential bounds in Section \ref{sec:exponential}, which is the main result of this paper.

We conclude this section by discussing techniques used in \cite{zhong2018near} for the small $\sigma$ regime, explaining why they fail in the high $\sigma$ regime, and how our approach connects with and diverges from theirs. In \cite{zhong2018near}, the main technical difficulty lies in controlling $\norm{W \hat{z}^\MLE}_\infty$, which is challenging due to the dependence between $W$ and $\hat{z}^\MLE$. To decouple this dependence, they construct a sequence of vectors $x^{(0)}, x^{(1)}, x^{(2)}, \ldots$, where $x^{(0)}$ is the leading eigenvector of the data matrix and the subsequent vectors are obtained by iteratively applying $F_1$. Additionally, they construct a leave-one-out counterpart that is independent of $W_{j\cdot}$: $x^{(0,-j)}, x^{(1,-j)}, x^{(2,-j)}, \ldots$. When $\sigma$ is small, they show the sequences satisfy:
\begin{align}
\ell_1(x^{(T)}, x^{(T,-j)}) &\lesssim \frac{1}{n}, \quad \norm{Wx^{(T)}}_\infty \lesssim \sqrt{n \log n}, \quad \ell_1(x^{(T)}, z^*) \lesssim 1,\label{eqn:main_sec_8}
\end{align}
for all $T \geq 0$, and the limit of $x^{(T)}$ is the MLE $\hat{z}^\MLE$. As a result, $\norm{W \hat{z}^\MLE}_\infty \lesssim \sqrt{n \log n}$ holds, which leads to the tightness of the SDP in the small $\sigma$ regime.

The key in \cite{zhong2018near}'s analysis is the $\ell_\infty$ norm result (\ref{eqn:main_sec_8}). Note that $x^{(T+1)} = F_1(x^{(T)})$ involves a column-wise normalization of $Y x^{(T)}$. As discussed in Section \ref{sec:2}, this normalization is difficult to analyze. However, once (\ref{eqn:main_sec_8}) holds, under the assumption that $\sigma$ is small, one can show the norm of each column $[Y x^{(T)}]_j$ is of the order $n$. Hence, the normalization $[Y x^{(T)}]_j / |[Y x^{(T)}]_j|$ is $[Y x^{(T)}]_j$ multiplied by a factor close to 1. This makes it easy to connect the error of $x^{(T+1)}$ with $x^{(T)}$. However, this argument no longer works when $\sigma$ is large, as the norms of columns $Y x^{(T)}$ are no longer guaranteed to be of the order $n$. For a column $[Y x^{(T)}]_j$ with a norm close to 0, its normalization $[Y x^{(T)}]_j / |[Y x^{(T)}]_j|$ differs dramatically from itself. This makes connecting the error of $x^{(T+1)}$ with that of $x^{(T)}$ difficult, leading to the failure of their analysis.

Compared to \cite{zhong2018near}'s analysis, our key novelty is in Step 1, where we bypass the normalization step of $F_1$ by replacing $F_1$ with a smoother mapping $G$. In Step 2, we follow the same idea of constructing leave-one-out sequences as in \cite{zhong2018near}. However, due to the use of $G$ instead of $F_1$, our analysis is much simpler, as we only need to establish (\ref{eqn:main_sec_7}) instead of the $\ell_\infty$ bound (\ref{eqn:main_sec_8}). Additionally, our sequences start from the ground truth $z^*$, instead of the eigenvector of the data matrix as in (\ref{eqn:main_sec_8}). This allows us to avoid the analysis of the eigenvector needed in \cite{zhong2018near}. It is also worth mentioning that \cite{zhong2018near} relates the SDP with the MLE through a dual certificate. We avoid this by simply using the fact that they are fixed points, which enables us to study not only the SDP but also the BM factorization more generally.

\subsection{Step 1: Approximate $\hat{z}^\MLE$ by Fixed Points of Lipschitz Mappings}\label{sec:step1}
As outlined in Section \ref{sec:high}, in this step, we are going to approximate  $\hat{z}^\MLE$ by fixed points of a Lipschitz mapping $G$ in order to upper bound $\frac{1}{n}\sum_{j \in [n]}\indic{\abs{[Y \hat{z}^\MLE]_j} < \delta n}$ by (\ref{eqn:main_sec_6}). Before introducing $g_t$ and $G$, we first introduce an auxiliary mapping $F_1'$ that is closely related to $F_1$ such that $\hat z^\MLE$ is its fixed point.

Recall the definition of the function $F_1$ in (\ref{eqn:F_1_def}). Note that for any $z\in\mathc_1^n$, $Yz = (z^*(z^*)^\H +\sigma W)z$ and consequently $[Yz]_j = z^*_j (z^*)^\H z + \sigma[Wz]_j$ for any $j\in[n]$. Then  (\ref{eqn:F_1_def}) can be written equivalently as
\begin{align*}
[F_1(z)]_j = \begin{cases}
\frac{ z^*_j (z^*)^\H z + \sigma [Wz]_j}{| z^*_j (z^*)^\H  z + \sigma [Wz]_j|},\text{ if } z^*_j (z^*)^\H  z + \sigma [Wz]_j\neq 0,\\
z_j,\text{ o.w.,}
\end{cases}\quad \forall j\in[n].
\end{align*}
Define a function $F'_1: \mathc^n\times \mathc\rightarrow\mathc^n$ such that for any $z\in\mathc_1^n$ and $s\in\mathc$, the $j$th coordinate of $F'_1(z)$ is
\begin{align}\label{eqn:F_1_prime_def}
[F'_1(z,s)]_j = \begin{cases}
\frac{ z^*_j s + \sigma [Wz]_j}{| z^*_j s + \sigma [Wz]_j|},\text{ if } z^*_j s + \sigma [Wz]_j\neq 0,\\
z_j,\text{ o.w.,}
\end{cases}\quad \forall j\in[n].
\end{align}
Since $\hat z^\MLE = F_1(\hat z^\MLE)$, it is easy to verify that $\hat z^\MLE$ is a fixed point of $F'(\cdot, (z^*)^\H \hat z^\MLE)$, i.e., $\hat z^\MLE = F_1'(\hat z^\MLE,(z^*)^\H \hat z^\MLE)$.
For simplicity, denote
\begin{align}\label{eqn:hat_s_def}
\hat s = (z^*)^\H\hat z^\MLE.
\end{align}
Then $z^\MLE$ satisfies
\begin{align}\label{eqn:MLE_F_1_prime_fixed}
\hat z^\MLE = F_1'(\hat z^\MLE,\hat s).
\end{align}
The difference between these two mappings $F_1(\cdot)$ and $F_1(\cdot,\hat s)$ is that: in $F_1(z)$, $z$ appears twice in its numerator $z^*_j (z^*)^\H z + \sigma [Wz]_j$; on the contrary, in $F_1'(z,\hat s)$, $z$ only appears once in its numerator $z^*_j \hat s + \sigma [Wz]_j$, despite that $\hat s$ depends on $z$. Hence, the mapping $F_1'(\cdot,\hat s)$ is less complicated than $F_1(\cdot)$ and is relatively easier to analyze.


Now we are ready to introduce $g_t$, an approximation of the normalization mapping $x\rightarrow x/|x|$, as we outline 
in Section \ref{sec:high}. For any $t\in\mathr$ such that $t>0$, define a function $g_t:\mathc \rightarrow\mathc$ as $g_t(x) = \frac{x}{\max\{|x|,t\}}$ for any $x\in\mathc$. That is, for any $x\in\mathc$,
\begin{align}\label{eqn:gt_def}
g_t(x) = \begin{cases}
\frac{x}{|x|},\text{ if }|x|>t ,\\
\frac{x}{t},\text{ o.w..}
\end{cases}
\end{align}
The following lemma shows $g_t$ is Lipschitz.
\begin{lemma}\label{lem:gt}
For any $x,y\in\mathc$, we have
$\abs{g_t(x) - g_t(y)} \leq \frac{\abs{x-y}}{t}, \forall t>0.$
\end{lemma}

With $g_t$, define $G:\mathcleq^n \times \mathc\times \mathr \rightarrow \mathcleq^n$ such that for any $z\in\mathcleq^n, s\in\mathc, t>0$, the $j$th coordinate of $G(z,s,t)$ is
\begin{align}\label{eqn:G_def}
[G(z,s,t)]_{j} = g_t(z^*_j s + \sigma [Wz]_j) = g_t([z^*s+\sigma Wz]_j), \forall j\in[n].
\end{align}
The following Lemma \ref{lem:G_propoerties} gives a list of properties $G$ has. First, since $g_t$ is Lipschitz, $G(\cdot,s,t)$ is also Lipschitz. With a suitable choice of $t$, it is a contraction mapping, and consequently has a unique fixed point which can be achieved by iteratively applying the function starting from $z^*$. The last property shows that the sensitivity of the fixed point with respect to $s$ is well-controlled.

\begin{lemma}\label{lem:G_propoerties}
The function $G(\cdot,\cdot,\cdot)$ has the following properties:
\begin{enumerate}
\item For any $x,y\in\mathc^n$ and for any $s\in\mathc,t>0$, we have
\begin{align*}
\norm{G(x,s,t) - G(y,s,t)} &\leq  t^{-1}\sigma \norm{W} \norm{x-y}.
\end{align*}
\item  For any $s \in\mathc,t\geq 2\sigma\norm{W}$, and for any $z^{(0)}\in\mathcleq^n$, define $z^{(T)} = G(z^{(T-1)},s,t)$ for all $T\in\mathbb{N}$. Then
\begin{align*}
\norm{z^{(T+1)} - z^{(T)}} &\leq \frac{1}{2}\norm{z^{(T)} - z^{(T-1)}}, \forall T\in\mathbb{N}.
\end{align*}
\item For any $s \in\mathc,t\geq 2\sigma\norm{W}$, $G(\cdot,s,t)$ has exactly one fixed point. That is, there exists one and only one $z\in\mathcleq^n$ such that $z= G(z,s,t)$. In addition, $z$ can be achieved by iteratively applying $G(\cdot,s,t)$ starting from $z^*$. That is, let $z^{(0)}=z^*$ and define $z^{(T)} = G(z^{(T-1)},s,t)$ for all $T\in\mathbb{N}$. We have $z=\lim_{T\rightarrow\infty}G(z^{(T)},s,t)$.
\item For any $s\in\mathc,s' \in\mathc,t\geq 2\sigma\norm{W}$, let $z$ be the fixed point of $G(\cdot,s,t)$ and let $z'$ be the fixed point of $G(\cdot,s',t)$. We have $\norm{z-z'}^2\leq 4nt^{-2}\abs{s-s'}^2$ and 
\begin{align*}
\norm{\br{z^* s + \sigma Wz} - \br{z^* s' + \sigma Wz'}}^2 \leq 4n\abs{s-s'}^2.
\end{align*}
\end{enumerate}
\end{lemma}

In addition, since $g_t$ approximates the normalization mapping $x\rightarrow x/|x|$, $G(\cdot,s,t)$ can be seen as a Lipschitz function that approximates $F_1'(\cdot, s)$. In fact, the approximation error can be bounded by the following lemma.
\begin{lemma}\label{lem:G_approximation_error}
For any $z\in\mathcleq^n$, $s\in\mathc$, and $t>0$, we have
\begin{align*}
\norm{F_1'(z, s) - G(z,s,t)}^2 \leq 4 \sum_{j\in[n]}\indic{|z^*_j s + \sigma [Wz]_j| < t}^2.
\end{align*}
\end{lemma}

With Lemmas \ref{lem:G_propoerties} and \ref{lem:G_approximation_error}, the following lemma shows that $\hat z^\MLE$ can be approximated by the fixed point of $G(\cdot,\hat s,t)$. 
\begin{lemma}\label{lem:MLE_G_fixed_points}
For any $t\geq 4\sigma \norm{W}$,  with $z\in\mathcleq^n$ being the fixed point of $G(\cdot,\hat s,t)$, we have
\begin{align*}
\norm{\hat z^\MLE - z}^2 &\leq 32\sum_{j\in[n]}  \indic{\abs{ z^*_j \hat s + \sigma [Wz]_j}<t}.
\end{align*}
\end{lemma}

With Lemma \ref{lem:MLE_G_fixed_points}, the quantity $\frac{1}{n}\sum_{j\in[n]}\indic{\abs{[Y\hat z^\MLE]_j}<\delta n}$ can be upper bounded by a similar quantity associated with the fixed point of $G$.
\begin{lemma}\label{lem:MLE_extreme_points_upper_bound}
For any $\delta\geq \frac{2\sigma \norm{W}}{n}$, with  $z\in\mathcleq^n$ being the fixed point of $G(\cdot,\hat s,2\delta n)$, we have
\begin{align*}
\frac{1}{n}\sum_{j\in[n]}\indic{\abs{[Y\hat z^\MLE]_j}<\delta n} \leq   \frac{9}{n}\sum_{j\in[n]}  \indic{\abs{z_j^*\hat s + \sigma[ W z]_j} < 2 \delta n}.
\end{align*}
\end{lemma}

Note that in Lemma \ref{lem:MLE_extreme_points_upper_bound}, the function $G(\cdot,\hat s,2dn)$ depends on $\hat s$ and consequently depend on $\hat z^\MLE$. We want to further decouple the dependence so that we only need to study $G(\cdot,s,t)$ for some fixed $s$. To achieve this, in Lemma \ref{lem:MLE_extreme_points_upper_bound_2}, we first show that  $G(\cdot,\hat s,2dn)$ can be replaced by $G(\cdot,|\hat s|,2dn)$. That is, the phase information in $G(\cdot,\hat s,2dn)$ is not important. What matters is its magnitude.

\begin{lemma}\label{lem:MLE_extreme_points_upper_bound_2}
For any $\delta\geq \frac{2\sigma \norm{W}}{n}$, with $z\in\mathcleq^n$ being the fixed point of $G(\cdot,|\hat s|,2\delta n)$, we have
\begin{align*}
\frac{1}{n}\sum_{j\in[n]}\indic{\abs{[Y\hat z^\MLE]_j}<\delta n} &\leq    \frac{9}{n}\sum_{j\in[n]}  \indic{\abs{z_j^*|\hat s| + \sigma \left[ W z \right]_j} < 2 \delta n}.
\end{align*}
\end{lemma}

Note that $|\hat s|$ is a real number. Once we have Lemma \ref{lem:MLE_extreme_points_upper_bound_2}, we then approximate $|\hat s|$ by points on a grid $\{s_0,s_1,s_2\ldots\}\subset \mathbb{R}$. Consequently, the fixed point of $G(\cdot,|\hat s|,2dn)$ can be approximated by those of $G(\cdot,s_k,2dn)$ where $k=0,1,2,\ldots$, leading to the following Proposition \ref{prop:MLE_extreme_points_upper_bound_grid}.

\begin{proposition}\label{prop:MLE_extreme_points_upper_bound_grid}
Suppose $\epsilon\in(0,1/2)$, $h>0$, and $\delta\geq \frac{2\sigma \norm{W}}{n}$. Assume $\hat z^\MLE$ satisfies $\ell_1(\hat z^\MLE,z^*)\leq \epsilon^2$. For each $k=0,1,2,\ldots, \lceil n\epsilon/h\rceil$, define $s_k = n- k h\in\mathr$  and let  $z_{s_k}\in\mathcleq^n$ be the fixed point of $G(\cdot, s_k,2\delta n)$. Then
\begin{align}
&\frac{1}{n}\sum_{j\in[n]}\indic{\abs{[Y\hat z^\MLE]_j}<\delta n}\nonumber \\
& \leq   9 \sum_{0\leq k\leq \lceil n\epsilon/h\rceil} \br{ \frac{1}{n}\sum_{j\in[n]}  \indic{\abs{z_j^*s_k + \sigma \left[ W z_{s_k} \right]_j} < 4 \delta n} }  +  \frac{9h^2}{\delta^2 n^2}\indic{h>\delta\sqrt{n}}.\label{eqn:main_sec_9}
\end{align}
\end{proposition}

Compared  Lemma \ref{lem:MLE_extreme_points_upper_bound_2}, Proposition \ref{prop:MLE_extreme_points_upper_bound_grid} avoids the appearance of $\hat s$ by using a grid $\{s_0,s_1,\ldots, s_{\lceil n\epsilon/h\rceil}\}$. Note that we can show $\hat s\in [(1-\epsilon)n,n]$ under the assumption $\ell_1(\hat z^\MLE,z^*)\leq \epsilon^2$. Hence, we only need to discretize the interval $[(1-\epsilon)n,n]$. In Proposition \ref{prop:MLE_extreme_points_upper_bound_grid}, $h$ is the distance among the points in the grid, a parameter can be optimized later. The use of the grid instead of $\hat s$ comes with costs, reflected in the two term in (\ref{eqn:main_sec_9}). Let $\hat k\in\{0,1,\ldots, \lceil n\epsilon/h\rceil\}$ be the index such that $s_{\hat k}$ is the one closest to $\hat s$ in the grid. First, note that $s_{\hat k}$ is still random as it depends on $\hat s$. To deal with, we upper bound the error associated with $s_{\hat k}$, $ \frac{1}{n}\sum_{j\in[n]}  \indic{\abs{z_j^*s_k + \sigma \left[ W z_{s_k} \right]_j} < 4 \delta n} $, by a summation of errors $ \sum_{0\leq k\leq \lceil n\epsilon/h\rceil} \br{ \frac{1}{n}\sum_{j\in[n]}  \indic{\abs{z_j^*s_k + \sigma \left[ W z_{s_k} \right]_j} < 4 \delta n} }  $, as each one is non-negative.
Second, the approximation error $|\hat s- s_{\hat k}|\leq h$ results in  the second term in (\ref{eqn:main_sec_9}). Nevertheless, the cost of having the summation of all indexes and the approximation error turns out to be negligible with a suitable choice of $h$.


The following Corollary \ref{cor:MLE_extreme_points_upper_bound_grid} simplifies the first term in (\ref{eqn:main_sec_9}) in order to make the analysis in Section \ref{sec:step2} easier.
For each $j\in[n]$, note that $\abs{z_j^*s_k + \sigma \left[ W z_{s_k} \right]_j} \geq \abs{z_j^*s_k} -  \abs{\sigma \left[ W z_{s_k} \right]_j} = s_k - \sigma \abs{\left[ W z_{s_k} \right]_j} $. Consequently, we have $\indic{\abs{z_j^*s_k + \sigma \left[ W z_{s_k} \right]_j} < 4 \delta n}    \leq \indic{s_k - \sigma \abs{\left[ W z_{s_k} \right]_j}  < 4\delta n}  = \indic{\sigma \abs{  \left[ W z_{s_k} \right]_j} > s_k- 4 \delta n}$, leading to the corollary.
\begin{corollary}\label{cor:MLE_extreme_points_upper_bound_grid}
Under the same conditions as in Proposition \ref{prop:MLE_extreme_points_upper_bound_grid}, we have
\begin{align*}
&\frac{1}{n}\sum_{j\in[n]}\indic{\abs{[Y\hat z^\MLE]_j}<\delta n} \\
& \leq   9 \sum_{0\leq k\leq \lceil n\epsilon/h\rceil} \br{ \frac{1}{n}\sum_{j\in[n]}  \indic{\sigma \abs{  \left[ W z_{s_k} \right]_j} > s_k- 4 \delta n} }  +  \frac{9h^2}{\delta^2 n^2}\indic{h>\delta\sqrt{n}}.
\end{align*}
\end{corollary}

With Corollary \ref{cor:MLE_extreme_points_upper_bound_grid}, we boil down the problem of upper bounding $\frac{1}{n}\sum_{j\in[n]}\indic{\abs{[Y\hat z^\MLE]_j}<\delta n}$ into a problem of analyzing the fixed point of $G(\cdot,s,t)$ for given $s,t$. Specifically, let $z$ be the fixed point, we want to analyze (\ref{eqn:main_sec_6}), which is the focus of the next section.

\subsection{Step 2: Leave-One-Out Analysis for Fixed Points of $G(\cdot,s,t)$}\label{sec:step2}
In the step, we are going to study the fixed point of $G(\cdot,s,t)$ to provide an upper bound for (\ref{eqn:main_sec_6}). As we outlined in Section \ref{sec:high}, the key is to decouple the dependence between $W$ and $z$ in the quantity $Wz$ using the leave-one-out technique.

For any $s \in\mathc,t>0$, define $z^{(0)}=z^*$ and
\begin{align}\label{eqn:zt_def}
z^{(T)} = G(z^{(T-1)},s,t),\forall T\in\mathbb{N}.
\end{align}
For any $j\in[n]$, define $W^{(-j)}\in\mathc^{n\times n}$ such that
\begin{align*}
W^{(-j)}_{k,l} =\begin{cases}
W_{k,l},\forall k\neq j \text{ and }l\neq j,\\
0,\text{ o.w..}
\end{cases}
\end{align*}
We refer $W^{(-j)}$ as a leave-one-out counterpart of $W$, as compared to $W$, it zeros out its $j$th column and row. As a result, it and its functions are independent of $W_{j\cdot}$.
Define $G^{(-j)}:\mathcleq^n \times \mathc\times \mathr \rightarrow \mathcleq^n$ such that for any $z\in\mathcleq^n$, the $j$th coordinate of $G^{(-j)}(z,s,t)$ is
\begin{align*}
[G^{(-j)}(z,s,t)]_j = g_t([z^*s + \sigma W^{(-j)} z]_j),\forall j\in[n].
\end{align*}
Define $z^{(0,-j)} = z^*$ and 
\begin{align}\label{eqn:ztj_def}
z^{(T,-j)} =  G^{(-j)}(z^{(T-1,-j)},s,t), \forall T\in\mathbb{N}.
\end{align}
That is, $G^{(-j)}(z,s,t)$ is a counterpart of $G(z,s,t)$ that uses $W^{(-j)}$ instead of $W$. Consequently, the sequence $\{z^{(T,-j)}\}_{T\geq 0}$ is independent of $W_{j\cdot}$. 

Note that the existence and uniqueness of the limit $z^{(\infty)}$ for the sequence $\{z^{(T)}\}_{T\geq 0}$ is guaranteed as long as $t\geq 2\sigma \norm{W}$, according to the properties of $G$ in Lemma \ref{lem:G_propoerties}. Similar properties hold for $G^{(-j)}$ (see Lemma \ref{lem:G_j_properties}), with which we can also show the existence and uniqueness of the limit for the sequence $\{z^{(T,-j)}\}_{T\geq 0}$ when $t\geq 2\sigma \norm{W}$. These lead to the following lemma about the limits and fixed points.
\begin{lemma}\label{lem:G_G_j_fixed_points}
For any $s\in\mathc$, $t\geq 2\sigma\norm{W}$, and $j\in[n]$, let $z\in\mathcleq^n$ be the fixed point of $G(\cdot,s,t)$ and let $z^{(-j)}\in\mathcleq^n$ be the fixed point of $G^{(-j)}(\cdot,s,t)$. Then
$z = \lim_{T\rightarrow\infty } z^{(T)},\text{ and }z^{(-j)}=\lim_{T\rightarrow\infty} z^{(T,-j)}.$
\end{lemma}

Note that $\norm{z^{(0)} - z^{(0,-j)}}=0$. With mathematical induction, the following lemma shows $z^{(T)}$ and $z^{(T,-j)}$ are  uniformly close for all $T\in\mathbb{N}$ and so are the limits.  %
\begin{lemma}\label{lem:leave_one_out_close}
Under the same conditions as in Lemma \ref{lem:G_G_j_fixed_points}, we have
$\norm{z^{(T)} - z^{(T,-j)}}\leq 3, \forall T\in\mathbb{N}.$
As a consequence,
$\norm{z - z^{(-j)}}\leq 3,\forall j\in[n].$
\end{lemma}

Note that both $z$ and $ z^{(-j)}$ are length-$n$ vectors in $\mathcleq^n$. Lemma \ref{lem:leave_one_out_close} means that they are pretty close to each other and consequently one can be approximated by the other one, leading to the following proposition.


\begin{proposition}\label{prop:extreme_points_leave_one_out}
Under the same conditions as in Lemma \ref{lem:G_G_j_fixed_points}, we have
\begin{align*}
\frac{1}{n} \sum_{j\in[n]}  \indic{\sigma \abs{[Wz]_j} \geq \abs{ s } - r} \leq  \frac{1}{n} \sum_{j\in[n]}\indic{\sigma \abs{W_{j\cdot} z^{(-j )}}  \geq \abs{ s } - r - 3\sigma \norm{W}},\quad\forall r\in\mathr.
\end{align*}
\end{proposition}
In Proposition \ref{prop:extreme_points_leave_one_out}, the left-hand and right-hand sides of the display correspond to (\ref{eqn:main_sec_6}) and (\ref{eqn:main_sec_10}), respectively. Note that for each $j\in[n]$, $z^{(-j)}$ is independent of $W_{j\cdot}$. In this way, we manage to decouple the dependence in $Wz$. The cost of replacing $z$ by $z^{(-j)}$ is $3\sigma \norm{W}$, which means the threshold in (\ref{eqn:main_sec_10}) is slightly smaller than that in (\ref{eqn:main_sec_6}).

\subsection{Exponential Bounds}\label{sec:exponential}
In Sections \ref{sec:step1} and \ref{sec:step2}, we carry out detailed analysis for the MLE to upper bound the quantity $\frac{1}{n}\sum_{j \in [n]}\indic{\abs{[Y \hat{z}^\MLE]_j} < \delta n}$, which is the main term appearing in $\ell_m(\hat V^{\BM,m},\hat z^\MLE)$, the distance between the MLE and the BM factorization in Corollary \ref{cor:fixed_points_MLE}. With Corollary \ref{cor:MLE_extreme_points_upper_bound_grid} and Proposition \ref{prop:extreme_points_leave_one_out}, Corollary \ref{cor:fixed_points_MLE} leads to the following lemma regarding $\ell_m(\hat V^{\BM,m},\hat z^\MLE)$.

\begin{lemma}\label{lem:deterministic}
Suppose $m\in\mathbb{N}\setminus\{1\}$, $\epsilon\in(0,1/2)$, $h>0$, and  $\delta \geq 2\sqrt{2}\br{6\epsilon + \frac{\sigma\norm{W}}{n}}$. Assume $ \ell_1(\hat z^\MLE,z^*)\leq \epsilon^2$ and $\ell_m(\hat V^{\BM,m},z^*)\leq \epsilon^2$ are satisfied. For each $k=0,1,2,\ldots, \lceil n\epsilon/h\rceil$, define $s_k = n- k h\in\mathr$ and let $z_{s_k}^{(-j)}\in\mathcleq^n$ be the fixed point of  $G^{(-j)}(\cdot,s_k,2\delta n)$ for each $j\in[n]$. Then we have
\begin{align*}
 &\ell_m(\hat V^{\BM,m},\hat z^\MLE)\\
 &\leq  72\sum_{0\leq k\leq \lceil n\epsilon/h\rceil} \br{ \frac{1}{n}\sum_{j\in[n]}  \indic{\sigma \abs{W_{j\cdot} z_{s_k}^{(-j)}} > \br{1-\epsilon-4\delta - \frac{3\sigma\norm{W}}{n}}n-h} }  +  \frac{72h^2}{\delta^2 n^2}\indic{h>\delta\sqrt{n}}.
\end{align*}
\end{lemma}
Lemma \ref{lem:deterministic} provides a deterministic upper bound for the difference between the MLE and the BM factorization, $\ell_m(\hat V^{\BM,m},\hat z^\MLE)$, as all the analysis carried out so far is completely deterministic.  Despite being  complicated, it makes it ready for us to obtain explicit expression for
using the fact that each entry of $W$ is Gaussian. To achieve this,
first note that Lemma \ref{lem:deterministic} involves $\norm{W}$ which requires an upper bound. It is known in the literature that  there exists some absolute constant $C_0>0$ such that
\begin{align}\label{eqn:W_opnorm_upper_bound}
\pbr{\norm{W}\leq C_0\sqrt{n}}\geq 1-n^{-10}.
\end{align}
Such concentration result is standard and is a direct consequence of Proposition 2.4 of \cite{rudelson2010non}. Regarding the display in Lemma \ref{lem:deterministic}, if we take expectations on both sides, then the indicator functions on its right-hand side become tail probabilities of Gaussian distributions, which are exponentially small. If we do not take expectations, then when $\sigma$ is small enough, all the indicator functions are equal to 0 with high probability, leading to the tightness $\ell_m(\hat V^{\BM,m},\hat z^\MLE)=0$. In both cases, we need to pick a suitable $h$ so that the second term in the display of the lemma is negligible. In this way, we obtain the following exponential bound in Theorem \ref{thm:upper_bound_Gaussian}. Theorem \ref{thm:intro2} is its immediate consequence.


\begin{theorem}\label{thm:upper_bound_Gaussian}
Suppose $m\in\mathbb{N}\setminus\{1\}$. There exist constants $C_1,C_2>0$ that only depend on $C_0$ such that:
\begin{enumerate}
\item When $\frac{n}{\sigma^2}\geq C_1$, we have 
\begin{align}\label{eqn:thm4_1}
\E \ell_m(\hat V^{\BM,m},\hat z^\MLE) \leq \ebr{-\frac{n}{8\sigma^2}} + n^{-10}.
\end{align}
\item When $\frac{n}{\sigma^2}\geq \max\{C_2,9\log n\}$, we have $\ell_m(\hat V^{\BM,m},\hat z^\MLE)=0$ with probability at least $1-n^{-1}$.
\end{enumerate}
\end{theorem}

In (\ref{eqn:thm4_1}), the term $n^{-10}$ comes from (\ref{eqn:W_opnorm_upper_bound}). Since in (\ref{eqn:W_opnorm_upper_bound}), by increasing $C_0$, we can replace $n^{-10}$ by $n^{-c}$ where $c>0$ can be sufficiently large, the $n^{-10}$ term in (\ref{eqn:thm4_1}) can be consequently replaced by the much smaller $n^{-c}$. Consequently, we view $n^{-10}$ as a negligible term in (\ref{eqn:thm4_1}) compared to its first term. Despite the upper bound in (\ref{eqn:thm4_1}) only holds for $\sigma$ such that $\frac{n}{\sigma^2}\geq C_1$, it can be restated so that it holds for all $\sigma>0$.
This is because if $\frac{n}{\sigma^2}< C_1$, due to the fact that $\E  \ell_m(\hat V^{\BM,m},\hat z^\MLE) $ is at most 1, it can be upper bounded by $\ebr{\frac{C_1}{8}}  \ebr{-\frac{n}{8\sigma^2}} $. Hence, there exists some constant $C_3>0$ such that $\E \ell_m(\hat V^{\BM,m},\hat z^\MLE) \leq C_3\ebr{-\frac{n}{8\sigma^2}} + n^{-10}$ for all $\sigma>0$. Once we establish it, by the connection between $ \ell_m(\hat V^{\BM,m},\hat z^\MLE)$ and $n^{-2}\|\hat Z^{\BM,m} - \hat z^\MLE (\hat z^\MLE)^\H\|_{\rm F}^2$ in (\ref{eqn:loss_connection}), we immediately establish (\ref{eqn:intro2_1}), the first part of  Theorem \ref{thm:intro2}.
With Markov inequality, the in-expectation upper bound (\ref{eqn:thm4_1}) can be converted into a in-probability upper bounded, resulting in the tightness result, the second part of Theorem \ref{thm:upper_bound_Gaussian}, which is also the second part of Theorem \ref{thm:intro2}.

To conclude this section, we reflect on the connections and distinctions between the analysis presented in this paper and our prior work \cite{gao2021exact, gao2022sdp}. These earlier studies focus on quantifying the deviations of the MLE and the SDP solutions from the ground truth $z^*$, leveraging the properties of fixed points. For instance, \cite{gao2022sdp} measures the distance between the SDP and $z^*$ through $\ell_n(\hat V^\SDP, z^*)$, and establishes that this distance is bounded by a certain quantity of $z^*$, similar to the results of Corollary \ref{cor:fixed_points_MLE}.
In contrast, the current work expands these frameworks to analyze $\ell_m(\hat V^{\BM,m},\hat z^\MLE)$. First, our analysis leverages fixed points of both mappings $F_m$ and $F_1$, instead of using a single mapping as in previous studies, to derive Corollary \ref{cor:fixed_points_MLE}. Second and more importantly, the current analysis confronts the challenge of handling $\frac{8}{n} \sum_{j \in [n]} \indic{\abs{[Y \hat{z}^\MLE]_j} < \delta n}$ within the upper bound  in the high-noise regime where $\sigma \gtrsim \sqrt{n/\log n}$, unlike in our earlier work where we only need to control a less complicated quantity of $z^*$. To tackle this, we have dedicated an entire Section \ref{sec:3} to this issue, introducing a novel approach that involves replacing $F_1$ with a Lipschitz mapping. This marks a significant methodological advancement over the more straightforward analysis conducted in our previous work.

The techniques developed in this paper may potentially extend to other low-rank optimization problems. One natural direction is orthogonal group synchronization, which generalizes phase synchronization by replacing phases with \(d\)-dimensional orthogonal matrices as latent variables. In this setting, analogues of the MLE, SDP, and BM factorization can be defined, along with corresponding fixed-point mappings \(F_1\) and \(F_m\). This structural similarity suggests that our analysis could, in principle, be adapted to this matrix-valued setting. Beyond synchronization problems, the core ideas might be also applicable to problems such as community detection in stochastic block models or \(k\)-means clustering. Exploring these directions remains an interesting topic for future research.

\section{Discussions}
This paper investigates  the distance between the MLE and the BM factorization in the phase synchronization, specifically focusing on the high-noise regime where $\sigma \gtrsim \sqrt{n/\log n}$. While this study contributes to our understanding of this specific problem, the BM factorization presents a wide array of interesting and important open questions, many of which are beyond the scope of our current analysis. Others present challenges that our current analytical framework is not equipped to solve. In the following discussions, we explore some of these unresolved issues, highlighting the challenges they present and suggesting potential avenues for future research.



\subsection{Computation and Optimization Landscape}\label{sec:computation}
This paper focuses on analyzing \(\hat{Z}^{\BM,m}\), the global solution to the BM factorization (\ref{eqn:BM_m_def}). A natural question is whether this solution can be computed efficiently in practice, especially given that the problem is non-convex. Interestingly, despite this non-convexity, the optimization landscape can sometimes be well-behaved. For example, it is known \cite{burer2005local, boumal2016non} that when \(m\) is large enough (\(m > \sqrt{2n}\)), all second-order critical points coincide with global ones. More recent works \cite{endor2024benign, ling2023local, waldspurger2020rank, mcrae2024benign, mcrae2024nonconvex} have shown that similar guarantees can still hold for smaller values of \(m\) when the noise level \(\sigma\) is not too large (\(\sigma \lesssim \sqrt{n/\log n}\)). In such cases, simple local algorithms like starting from the leading eigenvectors and applying the mapping \(F_m\) iteratively can reliably find the global solution \(\hat{Z}^{\BM,m}\).

However, when the problem becomes noisier ($\sigma \gtrsim \sqrt{{n}/{(\log n)}}$) and \(m\) is small, the landscape may no longer be so favorable.  In such scenarios, \cite{mei2017solving}  establishes upper bounds on the discrepancies between the global  and local maximum.  While our study is centered on $\hat{Z}^{\BM,m}$, the theoretical framework can be potentially extended to any fixed points of $F_m$ under certain conditions, as evidenced by the bounds developed in Theorem \ref{thm:fixed_points}. Nevertheless, the framework's applicability is limited and may not be further extended to second-order critical points due to its reliance on the specific properties of fixed points associated with $F_m$.

\subsection{Role of \( m \)}

Our main results hold for any \( m \geq 2 \), and interestingly, the bound we establish does not depend on \( m \). At first glance, this might seem surprising. In the BM factorization formulation (\ref{eqn:BM_m_def}), the size of the feasible set increases with \( m \). In fact, when \( m = 1 \), this set matches exactly with that of the MLE in (\ref{eqn:MLE_def_2}). Because the feasible set gets larger as \( m \) increases, one might intuitively expect the BM factorization solution \( \hat{Z}^{\BM,m} \) to deviate further from the MLE solution \( \hat{z}^{\MLE}(\hat{z}^{\MLE})^\H \). However, our theoretical bound in (\ref{eqn:intro2_1}) does not capture this potential dependence on \( m \).
This lack of dependence is not due to a fundamental property of the problem, but rather an artifact of our analysis. Our approach rewrites \( \hat{Z}^{\BM,m} \) as \( (\hat{V}^{\BM,m})^\H \hat{V}^{\BM,m} \), and relies on the fact that both \( \hat{V}^{\BM,m} \) and the MLE are fixed points of the mappings, \( F_m \) and \( F_1 \). We then show that these mappings behave like contractions, which allows us to bound the distance between their fixed points. However, in this analysis, the parameter \( m \) does not directly influence the contraction-type properties (see Lemmas \ref{lem:lm_contraction} and \ref{lem:normalization}), and therefore does not appear in Theorem \ref{thm:fixed_points} and subsequent analysis.

Incorporating \( m \) in the upper bound would require a more refined analysis that is beyond our current analysis framework. 
A possible direction for understanding the role of \(m\) is to build on recent advances in the analysis of optimization landscapes for synchronization problems \cite{ling2023local, mcrae2024nonconvex, endor2024benign, mcrae2025benign}. These works examine associated Laplacian matrices to derive conditions on $m$ under which the landscape is benign. Borrowing ideas from this line of work could provide valuable insights into how the choice of \(m\) influences the behavior of the BM factorization.

\subsection{Tightness of the SDP}

This paper does not introduce new theoretical bounds for the tightness of the SDP. Instead, we confirm that the condition \(\sigma \lesssim \sqrt{n/\log n}\) remains sufficient for ensuring tightness, consistent with earlier results by \cite{zhong2018near}. The broader question—under what conditions a non-convex problem can be exactly solved via a convex relaxation—is an important topic in optimization and continues to motivate a great deal of research.
In addition to phase synchronization, similar sufficient conditions for SDP tightness have been established in related problems such as orthogonal group synchronization \cite{won2022orthogonal, ling2020solving} and the generalized orthogonal Procrustes problem \cite{LING202362}. However, what remains less understood is whether such conditions are also necessary.  Addressing this question would be of substantial interest.

In the context of phase synchronization, \cite{bandeira2017tightness} provides some empirical insights. Specifically, its Figure 2 suggests that the SDP remains tight when \(\sigma \leq \sqrt{n}/3\). However, the accompanying text in \cite{bandeira2017tightness} interprets these results to suggest that  $\sigma$ might be allowed to grow at a rate of $\sqrt{n}/\text{polylog}(n)$  to maintain the tightness, where $\text{polylog}(n)$ means some polynomial in $\log n$. This interpretation is echoed in the follow-up work  \cite{zhong2018near}. There is a subtle discrepancy between the visual and textual implications drawn from these studies.

While our result in (\ref{eqn:intro2_1}) holds for the high-noise regime $\sigma \gtrsim \sqrt{n/\log n}$, it provides only an upper bound that may not be sharp. If it turns out to be sharp, it would imply that \(\sigma \lesssim \sqrt{n/\log n}\) is both a necessary and sufficient condition for tightness. To rigorously address this question, it would be important to prove that the SDP fails to be tight when \(\sigma\) exceeds a certain threshold, or to establish a lower bound on the discrepancy \(n^{-2}\|\hat{Z}^{\SDP} - \hat{z}^{\MLE}(\hat{z}^{\MLE})^\H\|_{\rm F}^2\). These remain open problems and are worthy of further investigation.

\subsection{$\sqrt{n}$ Regime}
Our results require that $\frac{n}{\sigma^2}$ must exceed a certain threshold, as reflected in Theorem \ref{thm:upper_bound_Gaussian}. Nonetheless, exploring the regime where $\sigma$ of the order $\sqrt{n}$ is also interesting and important.
 In this regime, \cite{javanmard2016phase} studies the asymptotic performance of the SDP in  synchronization problems. It reveals that the SDP is able to achieve a near-optimal performance for Bayesian estimation of $z^*$, using the cavity method from spin-glass theory. Our current analytical framework does not extend to this regime because we require $\frac{n}{\sigma^2}$ to be sufficiently large to establish the necessary contraction-type results in Section \ref{sec:2}.  To thoroughly understand the performance of the BM factorization in the $\sqrt{n}$ regime, further development of the methodologies presented in \cite{javanmard2016phase} may be required. 
 

\subsection{Eigenvector Method for Phase Synchronization}
While our study focuses on the SDP and the BM factorization for the phase synchronization, alternative methods such as the eigenvector approach \cite{singer2011angular} also merit consideration. The eigenvector method, which involves computing the leading eigenvector of the data matrix \(Y\) followed by entrywise normalization to ensure unit coordinates, is notable for its computational simplicity. In terms of performance, \cite{zhang2024exact} demonstrates that the eigenvector method achieves the estimation rate of \(\frac{\sigma^2}{n}\) for \(z^*\), which is comparable to that of the MLE and the SDP. 
However, evaluating how closely the eigenvector method approximates the MLE within our current analytical framework presents challenges. Unlike the BM factorization, the eigenvector method is not naturally compatible with our fixed-point analysis, which is central to our theory. As a result, it would require distinct analytical techniques that are outside the scope of this paper.

\section{Proofs}\label{sec:4} In this section, we give proofs of main results: Theorem \ref{thm:fixed_points}, 
Lemma \ref{lem:MLE_extreme_points_upper_bound}, Proposition \ref{prop:MLE_extreme_points_upper_bound_grid}, 
Proposition \ref{prop:extreme_points_leave_one_out}, and Theorem \ref{thm:upper_bound_Gaussian}. Due to the page limit, we include proofs of the remaining lemmas in the supplementary material.

\begin{proof}[Proof of Theorem \ref{thm:fixed_points}]
By the definition of $\ell_m$ in (\ref{eqn:lm_def}), there exists $a\in\mathc^m$ such that $\norm{a}=1$ and $\ell_m(V Y^\H, Yz) = n^{-1}\fnorm{V Y^\H - a(Yz)^\H}^2$.
From Lemma \ref{lem:lm_contraction}, we have
\begin{align}\label{eqn:fixed_points_proof_1}
\fnorm{V Y^\H - a(Yz)^\H}^2 = n\ell_m(V Y^\H, Yz) \leq n^3\br{6\epsilon + \frac{\sigma\norm{W}}{n}}^2 \ell_m(V,z).
\end{align}
On the other hand, 
we have
\begin{align}\label{eqn:fixed_points_proof_2}
\ell_m(F_m(V), F_1(z)) \leq n^{-1} \fnorm{F_m(V) - a\br{F_1(z)}^\H}^2 = n^{-1} \fnorm{F_m(V) - F_m(az^\H)}^2 ,
\end{align}
where the last equality is due to the fact that $F_m(az^\H) = a (F_1(z))^\H$.

Consider any $t>0$. For any $j\in[n]$, recall $V_j$, $[VY^\H]_j$, and $[F_m(V)]_{j}$ are the $j$th columns of $V$, $VY^\H$, and $F_m(V)$, respectively. Note that $[F_m(V)]_{j}$ and $[ F_m(az^\H)]_j$ can be written as
\begin{align*}
&[F_m(V)]_{j} = \frac{[VY^\H]_j}{\norm{[VY^\H]_j}} \indic{[VY^\H]_j \neq 0} + V_j \indic{[VY^\H]_j = 0}\\
\text{and }&[ F_m(az^\H)]_j = \frac{[(az^\H) Y^\H]_j}{\norm{[(az^\H) Y^\H]_j}} \indic{[(az^\H) Y^\H]_j \neq 0} +  a \bar{z_j} \indic{[(az^\H) Y^\H]_j  =0}.
\end{align*}
By applying (\ref{eqn:lem_normalization_2}) of Lemma \ref{lem:normalization}, we have 
\begin{align*}
\norm{[F_m(V)]_{j} - [ F_m(az^\H)]_j} &\leq   \frac{2\norm{[V Y^\H]_j - [a (Yz)^\H]_j}}{t} + 2\indic{\norm{[(az^\H)Y^\H]_j}<t} \\
&=  \frac{2\norm{[V Y^\H]_j - [a (Yz)^\H]_j}}{t} + 2\indic{ \abs{[Yz]_j} <t},
\end{align*}
where the last equality is due to $\norm{[(az^\H)Y^\H]_j} = \abs{[(Yz)^\H]_j} = \abs{[Yz]_j}$ as $\norm{a}=1$.

Summing over all $j\in[n]$, we have 
\begin{align*}
\fnorm{F_m(V) - F_m(az^\H)}^2 & = \sum_{j\in[n]}\norm{[F_m(V)]_{j} - [ F_m(az^\H)]_j}^2\\
&\leq \sum_{j\in[n]} \br{ \frac{2\norm{[V Y^\H]_j - [a (Yz)^\H]_j}}{t} + 2\indic{ \abs{[Yz]_j} <t}}^2 \\
&\leq  \sum_{j\in[n]} \br{\frac{4\norm{[V Y^\H]_j - [a (Yz)^\H]_j}^2}{t^2} + 4\indic{ \abs{[Yz]_j} <t}}\\
& = 4t^{-2} \fnorm{VY^\H - a (Yz)^\H}^2 + 4  \sum_{j\in[n]}\indic{ \abs{[Yz]_j} <t}.
\end{align*}
Together with (\ref{eqn:fixed_points_proof_1}) and (\ref{eqn:fixed_points_proof_2}), we have
\begin{align*}
\ell_m(F_m(V), F_1(z))&\leq n^{-1}\br{ 4t^{-2} \fnorm{VY^\H - a (Yz)^\H}^2 + 4  \sum_{j\in[n]}\indic{ \abs{[Yz]_j} <t}}\\
&\leq n^{-1}\br{ 4t^{-2}  n^3\br{6\epsilon + \frac{\sigma\norm{W}}{n}}^2 \ell_m(V,z) + 4  \sum_{j\in[n]}\indic{ \abs{[Yz]_j} <t}}\\
&\leq \frac{4n^2}{t^2}\br{6\epsilon + \frac{\sigma\norm{W}}{n}}^2 \ell_m(V,z) + \frac{4}{n}  \sum_{j\in[n]}\indic{ \abs{[Yz]_j} <t},
\end{align*}
which proves (\ref{eqn:main_sec_3}).

To prove (\ref{eqn:main_sec_4}), set $t=\delta n$. Since $z = F_1(z)$ and $V = F_m(V)$, we have $\ell_m(V,z) = \ell_m(F_m(V), F_1(z))$, and  the above display can be written as
\begin{align*}
\br{1-\frac{4}{\delta^2}\br{6\epsilon + \frac{\sigma\norm{W}}{n}}^2} \ell_m(V,z)\leq \frac{4}{n}  \sum_{j\in[n]}\indic{ \abs{[Yz]_j} <\delta n}.
\end{align*}
Since $\frac{4}{\delta^2}\br{6\epsilon + \frac{\sigma\norm{W}}{n}}^2 \leq 1/2$, we have 
$ \ell_m(V,z)\leq \frac{8}{n}  \sum_{j\in[n]}\indic{ \abs{[Yz]_j} <\delta n}.$
\end{proof}

\begin{proof}[Proof of Lemma \ref{lem:MLE_extreme_points_upper_bound}]
For any $t\geq 4\sigma\norm{W}$, let $z\in\mathcleq^n$ be the fixed point of $G(\cdot,\hat s,t)$. By Lemma \ref{lem:MLE_G_fixed_points}, we have
$\norm{\hat z^\MLE - z}^2 \leq 32\sum_{j\in[n]}  \indic{\abs{ z^*_j \hat s + \sigma [Wz]_j}<t}.$
For any $j\in[n]$, note that
\begin{align*}
&[Y\hat z^\MLE]_j = [(z^*(z^*)^\H + \sigma W) \hat z^\MLE]_j = z_j^* \hat s +  \sigma [W \hat z^\MLE]_j.
\end{align*}
and $[z^*\hat s + \sigma W z]_j = z^*_j \hat s + \sigma [W z]_j$. Then $[Y\hat z^\MLE]_j - [z^*\hat s + \sigma W z]_j = \sigma [W (\hat z^\MLE - z)]_j$. Hence,
\begin{align*}
\sum_{j\in[n]}\abs{[Y\hat z^\MLE]_j - [z^*\hat s + \sigma W z]_j }^2&= \sum_{j\in[n]} \abs{ \sigma [W (\hat z^\MLE - z)]_j}^2 = \sigma^2 \norm{W (\hat z^\MLE - z)}^2 \\
& \leq \sigma^2 \norm{W}^2 \norm{\hat z^\MLE -z}^2.
\end{align*}
Then, for any $r>0$, we have
\begin{align*}
\sum_{j\in[n]} \indic{\abs{[Y\hat z^\MLE]_j - [z^*\hat s + \sigma W z]_j } >  r }&\leq  r ^{-2}\sum_{j\in[n]}\abs{[Y\hat z^\MLE]_j - [z^*\hat s + \sigma W z]_j }^2 \\
&\leq   r ^{-2}\sigma^2 \norm{W}^2 \norm{\hat z^\MLE -z}^2.
\end{align*}
Note that for any $a,b\in\mathr$, we have
\begin{align*}
\indic{\abs{a}<  r } &= \indic{\abs{a-b+b}<  r } \leq \indic{\abs{b}-\abs{a-b} <  r } \\
&=\indic{\abs{b}-\abs{a-b} <  r ,\; \abs{a-b}> r } + \indic{\abs{b}-\abs{a-b} <  r ,\; \abs{a-b}\leq  r }\\
&\leq \indic{ \abs{a-b}> r } +\indic{\abs{b}< 2r}.\numberthis\label{eqn:indicator_split}
\end{align*}
Hence, for each $j\in[n]$,
\begin{align*}
\indic{\abs{[Y\hat z^\MLE]_j}< r } &\leq  \indic{\abs{[Y\hat z^\MLE]_j - [z^*\hat s + \sigma W z]_j } >  r } + \indic{\abs{[z^*\hat s + \sigma W z]_j} < 2r}.
\end{align*}
Summing over all $j\in[n]$, we have
\begin{align*}
\sum_{j\in[n]} \indic{\abs{[Y\hat z^\MLE]_j}< r }  &\leq  \sum_{j\in[n]} \indic{\abs{[Y\hat z^\MLE]_j - [z^*\hat s + \sigma W z]_j } >  r } +\sum_{j\in[n]}  \indic{\abs{[z^*\hat s + \sigma W z]_j} < 2r}\\
&\leq  r ^{-2}\sigma^2 \norm{W}^2 \norm{\hat z^\MLE -z}^2 + \sum_{j\in[n]}  \indic{\abs{[z^*\hat s + \sigma W z]_j} < 2r}\\
&=  r ^{-2}\sigma^2 \norm{W}^2 \norm{\hat z^\MLE -z}^2 + \sum_{j\in[n]}  \indic{\abs{z_j^*\hat s + \sigma[ W z]_j} < 2r}.
\end{align*}

As a consequence,
\begin{align*}
\sum_{j\in[n]} \indic{\abs{[Y\hat z^\MLE]_j}<r}  &\leq 32r^{-2}\sigma^2 \norm{W}^2  \sum_{j\in[n]}  \indic{\abs{ z^*_j \hat s + \sigma [Wz]_j}<t} +  \sum_{j\in[n]}  \indic{\abs{z_j^*\hat s + \sigma[ W z]_j} < 2r}.
\end{align*}
Consider any $\delta \geq  \frac{2\sigma \norm{W}}{n}$. Set $t=2\delta n$ and $r=\delta n$, we have 
\begin{align*}
\sum_{j\in[n]} \indic{\abs{[Y\hat z^\MLE]_j}<\delta n}   & \leq  \frac{32\sigma^2\norm{W}^2}{\delta^2 n^2}   \sum_{j\in[n]}  \indic{\abs{ z^*_j \hat s + \sigma [Wz]_j}<2\delta n} +  \sum_{j\in[n]}  \indic{\abs{z_j^*\hat s + \sigma[ W z]_j} < 2 \delta n}\\
&\leq  9 \sum_{j\in[n]}  \indic{\abs{z_j^*\hat s + \sigma[ W z]_j} < 2 \delta n}.
\end{align*}
Multiplying $n^{-1}$ on both sides, we complete the proof.
\end{proof}

\begin{proof}[Proof of Proposition \ref{prop:MLE_extreme_points_upper_bound_grid}]
Let $a\in\mathc_1$ such that $ \norm{\hat z^\MLE - az^*}^2 = n\ell_1(\hat z^\MLE,z^*) \leq n\epsilon^2$. Then
\begin{align*}
\abs{\hat s} &=\abs{(z^*)^\H \hat z^\MLE} = \abs{(z^*)^\H (\hat z^\MLE-az^*) + (z^*)^\H(az^*)} \geq \abs{(z^*)^\H(az^*)} - \abs{(z^*)^\H (\hat z^\MLE-az^*)}\\
&\geq  n - \sqrt{n}\norm{\hat z^\MLE-az^*} \geq (1-\epsilon)n.
\end{align*}
Then $\abs{\hat s}\in [(1-\epsilon)n,n]\subset[s_{\lceil n\epsilon/h\rceil},s_0]$. Define $\hat k = \argmin_{0\leq k \leq \lceil n\epsilon/h\rceil} \abs{\abs{\hat s}  - s_{k}}$. Then $\abs{\abs{\hat s} - s_{\hat k}} \leq h$.

Consider any $\delta\geq \frac{2\sigma \norm{W}}{n}$ and let $z\in\mathcleq^n$ be the fixed point of $G(z,|\hat s|,2\delta n)$.  By the 4th property of Lemma \ref{lem:G_propoerties}, we have
\begin{align*}
 \norm{\br{z^* \abs{\hat s} + \sigma Wz} - \br{z^* s_{\hat k} + \sigma W z_{s_{\hat k}}}  }^2 &\leq  4n\abs{\abs{\hat s} - s_{\hat k}}^2\leq 4nh^2.
\end{align*}
Note that we have the following fact: $\sum_{j\in[n]}\indic{\abs{x_j}>t}\leq t^{-2}\norm{x}^2\indic{\norm{x}>t}$ for any $x\in\mathr^n$ and any $t>0$. This is because if $\norm{x}\leq t$, then $\sum_{j\in[n]}\indic{\abs{x_j}>t} =0$; if $\norm{x}>t$, then $\sum_{j\in[n]}\indic{\abs{x_j}>t} \leq \sum_{j\in[n]}t^{-2}|x_j|^2\indic{\abs{x_j}>t} \leq \sum_{j\in[n]}t^{-2}|x_j|^2 = t^{-2}\norm{x}^2$. Hence,
\begin{align*}
&\sum_{j\in[n]} \indic{ \abs{\sbr{z^* \abs{\hat s} + \sigma Wz}_j - \sbr{z^* s_{\hat k} + \sigma W z_{s_{\hat k}}}_j} > 2\delta n} \\
&\leq  (2\delta n)^{-2} \norm{\br{z^* \abs{\hat s} + \sigma Wz} - \br{z^* s_{\hat k} + \sigma W z_{s_{\hat k}}}  }^2 \indic{ \norm{\br{z^* \abs{\hat s} + \sigma Wz} - \br{z^* s_{\hat k} + \sigma W z_{s_{\hat k}}}  } >2\delta n}\\
&\leq (2\delta n)^{-2} \br{4nh^2} \indic{\sqrt{4nh^2}>2\delta n}= \frac{h^2}{\delta^2n} \indic{h>\delta\sqrt{n}}.
\end{align*}

Using (\ref{eqn:indicator_split}) and the above display, we have 
\begin{align*}
&\frac{1}{n}\sum_{j\in[n]}  \indic{\abs{z_j^*|\hat s| + \sigma \left[ W z \right]_j} < 2 \delta n} \\
&= \frac{1}{n}\sum_{j\in[n]}  \indic{\abs{z_j^*s_{\hat k} + \sigma \left[ W z_{s_{\hat k}} \right]_j} < 4 \delta n}  + \frac{1}{n} \sum_{j\in[n]} \indic{ \abs{\sbr{z^* \abs{\hat s} + \sigma Wz}_j - \sbr{z^* s_{\hat k} + \sigma W z_{s_{\hat k}}}_j} > 2\delta n}\\
&\leq \frac{1}{n}\sum_{j\in[n]}  \indic{\abs{z_j^*s_{\hat k} + \sigma \left[ W z_{s_{\hat k}} \right]_j} < 4 \delta n}  + \frac{h^2}{\delta^2 n^2}\indic{h>\delta\sqrt{n}}\\
&\leq  \sum_{0\leq k\leq \lceil n\epsilon/h\rceil} \br{\frac{1}{n}\sum_{j\in[n]}  \indic{\abs{z_j^*s_k + \sigma \left[ W z_{s_k} \right]_j} < 4 \delta n}}  +  \frac{h^2}{\delta^2 n^2}\indic{h>\delta\sqrt{n}},
\end{align*}
where the last inequality is due to $\hat k\in\{0,1,2,\ldots, \lceil n\epsilon/h\rceil \}$. By applying Lemma \ref{lem:MLE_extreme_points_upper_bound_2}, we have
\begin{align*}
&\frac{1}{n}\sum_{j\in[n]}\indic{\abs{[Y\hat z^\MLE]_j}<\delta n} \\
 &\leq   \frac{9}{n}\sum_{j\in[n]}  \indic{\abs{z_j^*|\hat s| + \sigma \left[ W z \right]_j} < 4 \delta n}\\
& \leq 9 \br{\sum_{0\leq k\leq \lceil n\epsilon/h\rceil} \br{ \frac{1}{n}\sum_{j\in[n]}  \indic{\abs{z_j^*s_k + \sigma \left[ W z_{s_k} \right]_j} < 4 \delta n} } +  \frac{h^2}{\delta^2 n^2}\indic{h>\delta\sqrt{n}}}\\
&\leq  9 \sum_{0\leq k\leq \lceil n\epsilon/h\rceil} \br{ \frac{1}{n}\sum_{j\in[n]}  \indic{\abs{z_j^*s_k + \sigma \left[ W z_{s_k} \right]_j} < 4 \delta n} } +\frac{9h^2}{\delta^2 n^2}\indic{h>\delta\sqrt{n}}.
\end{align*}
\end{proof}

\begin{proof}[Proof of Proposition \ref{prop:extreme_points_leave_one_out}]
For any $j\in[n]$, we have
\begin{align*}
 \abs{[W z ]_j} & = \abs{[ W z^{( -j )}]_j + [ W (  z  - z^{( -j )})]_j}=\abs{W_{j\cdot} z^{( -j )} + [ W (  z  - z^{( -j )})]_j}\\
 &\leq \abs{W_{j\cdot} z^{( -j )}} +   \abs{[ W (  z  - z^{( -j )})]_j}\leq \abs{W_{j\cdot} z^{( -j )}} +  \norm{W (  z  - z^{( -j )})}\\
 &\leq \abs{W_{j\cdot} z^{( -j )}}  +  \norm{W}\norm{(  z  - z^{( -j )})} \leq \abs{W_{j\cdot} z^{( -j )}}  + 3 \norm{W},
\end{align*}
where the last inequality is due to Lemma \ref{lem:leave_one_out_close}.
Hence, for any $r\in\mathr$, $\indic{\sigma \abs{[Wz]_j} \geq \abs{ s } - r} \leq \indic{\sigma \abs{W_{j\cdot} z^{( -j )}}  \geq \abs{ s } - r - 3\sigma \norm{W}}$. Summing over all $j\in[n]$, we have
\begin{align*}
\frac{1}{n} \sum_{j\in[n]}  \indic{\sigma \abs{[Wz]_j} \geq \abs{ s } - r} \leq  \frac{1}{n} \sum_{j\in[n]}\indic{\sigma \abs{W_{j\cdot} z^{( -j )}}  \geq \abs{ s } - r - 3\sigma \norm{W}}.
\end{align*}
\end{proof}

\begin{proof}[Proof of Theorem \ref{thm:upper_bound_Gaussian}]
Consider any $m\in\mathbb{N}\setminus\{1\}$. Recall the upper bound for $\norm{W}$ in (\ref{eqn:W_opnorm_upper_bound}).
We first prove (\ref{eqn:thm4_1}). Since $$ \ell_m(V,\hat z^\MLE) =  \ell_m(V,\hat z^\MLE) \indic{\norm{W}\leq C_0\sqrt{n}} + \ell_m(V,\hat z^\MLE) \indic{\norm{W}> C_0\sqrt{n}},$$ we have
\begin{align*}
\E \ell_m(V,\hat z^\MLE) &= \E \br{ \ell_m(V,\hat z^\MLE) \indic{\norm{W}\leq C_0\sqrt{n}} }+   \E \br{ \ell_m(V,\hat z^\MLE) \indic{\norm{W}> C_0\sqrt{n}}}  \\
&\leq  \E \br{\ell_m(V,\hat z^\MLE) \indic{\norm{W}\leq C_0\sqrt{n}}} + \E\indic{\norm{W}> C_0\sqrt{n}}  \\
&\leq   \E \br{\ell_m(V,\hat z^\MLE) \indic{\norm{W}\leq C_0\sqrt{n}}} + n^{-10}, \numberthis\label{eqn:upper_bound_Gaussian_proof_1}
\end{align*}
where the last inequality is due to (\ref{eqn:W_opnorm_upper_bound}),
we focus on analyzing $\E \br{ \ell_m(V,\hat z^\MLE) \indic{\norm{W}\leq C_0\sqrt{n}}}$.

%

Assume $\norm{W}\leq C_0\sqrt{n}$. We first perform some deterministic analysis on $ \ell_m(V,\hat z^\MLE)$.
From Lemma \ref{lem:MLE_SDP_Loose_bound}, we have $\ell_1(\hat z^\MLE,z^*),\ell_m(\hat V^{\BM,m},z^*)\leq 8\sigma \norm{W}/n \leq 8C_0\sigma/\sqrt{n}$. Set $\epsilon =\br{\frac{8C_0\sigma}{\sqrt{n}}}^\frac{1}{2}$ so that $\ell_1(\hat z^\MLE,z^*),\ell_m(\hat V^{\BM,m},z^*)\leq \epsilon^2$. Assume 
$\frac{C_0 \sigma}{\sqrt{n}}<\frac{1}{16}$. Then $\epsilon <1/2$ and
\begin{align*}
2\sqrt{2}\br{6\epsilon + \frac{\sigma\norm{W}}{n}} &= 2\sqrt{2}\br{6\br{\frac{8C_0\sigma}{\sqrt{n}}}^\frac{1}{2} + \frac{\sigma\norm{W}}{n}}  = 48 \br{\frac{C_0\sigma}{\sqrt{n}}}^\frac{1}{2}  + 2\sqrt{2} \frac{C_0\sigma}{\sqrt{n}}\leq 49 \br{\frac{C_0\sigma}{\sqrt{n}}}^\frac{1}{2}.
\end{align*}
Hence, by setting $\delta =49 \br{\frac{C_0\sigma}{\sqrt{n}}}^\frac{1}{2}$, we have $\delta\geq 2\sqrt{2}\br{6\epsilon + \frac{\sigma\norm{W}}{n}}$. In this way, the conditions required in  Lemma \ref{lem:deterministic} are satisfied, which leads to
\begin{align*}
 &\ell_m(\hat V^{\BM,m},\hat z^\MLE)\\
 &\leq  72 \sum_{0\leq k\leq \lceil n\epsilon/h\rceil} \br{ \frac{1}{n}\sum_{j\in[n]}  \indic{\sigma \abs{W_{j\cdot} z_{s_k}^{(-j)}} > \br{1-\epsilon-4\delta - \frac{3\sigma\norm{W}}{n}}n-h} }  +  \frac{72h^2}{\delta^2 n^2}\indic{h>\delta\sqrt{n}}\\
 &\leq 72\sum_{0\leq k\leq \lceil n\epsilon/h\rceil} \br{ \frac{1}{n}\sum_{j\in[n]}  \indic{\sigma \abs{W_{j\cdot} z_{s_k}^{(-j)}} > \br{1-\epsilon-4\delta - \frac{3C_0\sigma}{\sqrt{n}}}n-h} }  +  \frac{72h^2}{\delta^2 n^2}\indic{h>\delta\sqrt{n}}.
\end{align*}
Since the above inequality holds under the assumption $\norm{W}\leq C_0\sqrt{n}$, we can write it as
\begin{align*}
 &\ell_m(\hat V^{\BM,m},\hat z^\MLE)\indic{\norm{W}\leq C_0\sqrt{n}}\\
 &\leq  72 \sum_{0\leq k\leq \lceil n\epsilon/h\rceil} \br{ \frac{1}{n}\sum_{j\in[n]}  \indic{\sigma \abs{W_{j\cdot} z_{s_k}^{(-j)}} > \br{1-\epsilon-4\delta - \frac{3C_0\sigma}{\sqrt{n}}}n-h} }  +  \frac{72h^2}{\delta^2 n^2}\indic{h>\delta\sqrt{n}}. \numberthis\label{eqn:thm4_proof_3}
\end{align*}

Take expectation on both sides of (\ref{eqn:thm4_proof_3}). Then, we have
\begin{align*}
&\E \ell_m(\hat V^{\BM,m},\hat z^\MLE)\indic{\norm{W}\leq C_0\sqrt{n}}\\
 &\leq  72 \sum_{0\leq k\leq \lceil n\epsilon/h\rceil} \br{ \frac{1}{n}\sum_{j\in[n]}  \pbr{\sigma \abs{W_{j\cdot} z_{s_k}^{(-j)}} > \br{1-\epsilon-4\delta - \frac{3C_0\sigma}{\sqrt{n}}}n-h} }  +  \frac{72h^2}{\delta^2 n^2}\indic{h>\delta\sqrt{n}}\\
 &\leq  72 \sum_{0\leq k\leq \lceil n\epsilon/h\rceil} \br{ \frac{1}{n}\sum_{j\in[n]}  \pbr{\sigma \abs{W_{j\cdot} z_{s_k}^{(-j)}} > \br{1-\epsilon-4\delta - \frac{3C_0\sigma}{\sqrt{n}}}n-h} }  +  \frac{72h^2}{\delta^2 n^2}.
\end{align*}
For each $j$ and $k$, due to the independence between $W_{j\cdot}$ and $z_{s_k}^{(-j)}$, $W_{j\cdot}z_{s_k}^{(-j)}$ is complex Gaussian with zero mean and variance $\|{z_{s_k}^{(-j)}}\|^2\leq n$. Then $|{W_{j\cdot}z_{s_k}^{(-j)}}|$ is Gaussian with zero mean and variance $\|{z_{s_k}^{(-j)}}\|^2$. Let $1-\Phi(x)$ be cumulative distribution function of the standard normal. That is, $\Phi(x) = \int_{u\geq x}^\infty 1/\sqrt{2\pi} \ebr{-u^2/2}du$. Then
\begin{align*}
 \pbr{\sigma \abs{W_{j\cdot} z_{s_k}^{(-j)}} > \br{1-\epsilon-4\delta - \frac{3C_0\sigma}{\sqrt{n}}}n-h} &= \Phi\br{\frac{1}{\sigma \norm{z_{s_k}^{(-j)}}} \br{1-\epsilon-4\delta - \frac{3C_0\sigma}{\sqrt{n}} - \frac{h}{n}}n}\\
 &\leq \Phi\br{ \br{1-\epsilon-4\delta - \frac{3C_0\sigma}{\sqrt{n}} - \frac{h}{n}}\frac{\sqrt{n}}{\sigma }},
\end{align*} 
which is invariant of $j$ or $k$. As a result,
\begin{align*}
&\E \ell_m(\hat V^{\BM,m},\hat z^\MLE)\indic{\norm{W}\leq C_0\sqrt{n}}\\
 &\leq  72 \sum_{0\leq k\leq \lceil n\epsilon/h\rceil}  \Phi\br{ \br{1-\epsilon-4\delta - \frac{3C_0\sigma}{\sqrt{n}} - \frac{h}{n}}\frac{\sqrt{n}}{\sigma }}+  \frac{72h^2}{\delta^2 n^2}\\
 &\leq 72  \left\lceil \frac{n\epsilon}{h}\right\rceil\Phi\br{ \br{1-\epsilon-4\delta - \frac{3C_0\sigma}{\sqrt{n}} - \frac{h}{n}}\frac{\sqrt{n}}{\sigma }}+  \frac{72h^2}{\delta^2 n^2}.\numberthis\label{eqn:thm4_proof_4}
\end{align*}
Recall that $\epsilon =\br{\frac{8C_0\sigma}{\sqrt{n}}}^\frac{1}{2}$ and $\delta =49 \br{\frac{C_0\sigma}{\sqrt{n}}}^\frac{1}{2}$. 
Set $h= n\ebr{-\frac{n}{8\sigma^2}}$. 
Then (\ref{eqn:thm4_proof_4}) becomes
\begin{align*}
&\E \ell_m(\hat V^{\BM,m},\hat z^\MLE)\indic{\norm{W}\leq C_0\sqrt{n}}\\
&\leq 72 \ebr{\frac{n}{8\sigma^2}} \Phi\br{\br{1-\br{(196+2\sqrt{2})\br{\frac{C_0\sigma}{\sqrt{n}}}^\frac{1}{2} + \frac{3C_0\sigma}{\sqrt{n}} +  \ebr{-\frac{n}{8\sigma^2}}}}\frac{\sqrt{n}}{\sigma}}\\
&\quad  + \frac{1}{100C_0} \br{\frac{n}{\sigma^2}}^\frac{1}{2} \ebr{- \frac{n}{4\sigma^2}}.
\end{align*}
Note that $\Phi(x)\leq \frac{2}{\sqrt{\pi}}\ebr{-\frac{x^2}{2}}$ for any $x>0$. 
Then there exists some constant $C_1>0$ that only depends on $C_0$, such that if $\frac{n}{\sigma^2}\geq C_1$, we have
\begin{align}\label{eqn:thm4_proof_2}
(196+2\sqrt{2})\br{\frac{C_0\sigma}{\sqrt{n}}}^\frac{1}{2} + \frac{3C_0\sigma}{\sqrt{n}} +  \ebr{-\frac{n}{8\sigma^2}} \leq \frac{1}{4},
\end{align}
and consequently,
\begin{align*}
&\E \ell_m(\hat V^{\BM,m},\hat z^\MLE)\indic{\norm{W}\leq C_0\sqrt{n}}\\
&\leq 72 \ebr{\frac{n}{8\sigma^2}} \Phi\br{ \frac{3\sqrt{n}}{4\sigma}}+ \frac{1}{100C_0} \br{\frac{n}{\sigma^2}}^\frac{1}{2} \ebr{- \frac{n}{4\sigma^2}}\\
&\leq 72 \ebr{\frac{n}{8\sigma^2}}  \frac{2}{\sqrt{\pi}}\ebr{ -\frac{9n}{32\sigma^2}} +  \frac{1}{100C_0} \br{\frac{n}{\sigma^2}}^\frac{1}{2} \ebr{- \frac{n}{4\sigma^2}}\\
&= \frac{144}{\sqrt{\pi}} \ebr{-\frac{5n}{32\sigma^2}} + 2 \ebr{- \frac{n}{6\sigma^2}}\leq \ebr{-\frac{n}{8\sigma^2}}.
\end{align*}
The proof for (\ref{eqn:thm4_1}) is complete with (\ref{eqn:upper_bound_Gaussian_proof_1}).

To prove the second part of the theorem, note that from (\ref{eqn:thm4_proof_3}), we can also get a in-probability bound. Set $h= \delta\sqrt{n}$. Then  (\ref{eqn:thm4_proof_3}) becomes
\begin{align*}
 &\ell_m(\hat V^{\BM,m},\hat z^\MLE)\indic{\norm{W}\leq C_0\sqrt{n}}\\
 &\leq  72 \sum_{0\leq k\leq \lceil n\epsilon/h\rceil} \br{ \frac{1}{n}\sum_{j\in[n]}  \indic{\sigma \abs{W_{j\cdot} z_{s_k}^{(-j)}} > \br{1-\epsilon-4\delta - \frac{3C_0\sigma}{\sqrt{n}}}n-h} }. 
\end{align*}
Using Markov inequality, we have
\begin{align*}
&\pbr{\ell_m(\hat V^{\BM,m},\hat z^\MLE)\indic{\norm{W}\leq C_0\sqrt{n}} >0} \\
&\leq \pbr{72 \sum_{0\leq k\leq \lceil n\epsilon/h\rceil} \br{ \frac{1}{n}\sum_{j\in[n]}  \indic{\sigma \abs{W_{j\cdot} z_{s_k}^{(-j)}} > \br{1-\epsilon-4\delta - \frac{3C_0\sigma}{\sqrt{n}}}n-h} }>0}\\
&= \pbr{ \sum_{0\leq k\leq \lceil n\epsilon/h\rceil} \br{ \frac{1}{n}\sum_{j\in[n]}  \indic{\sigma \abs{W_{j\cdot} z_{s_k}^{(-j)}} > \br{1-\epsilon-4\delta - \frac{3C_0\sigma}{\sqrt{n}}}n-h} }>0}\\
&= \pbr{ \sum_{0\leq k\leq \lceil n\epsilon/h\rceil} \br{ \frac{1}{n}\sum_{j\in[n]}  \indic{\sigma \abs{W_{j\cdot} z_{s_k}^{(-j)}} > \br{1-\epsilon-4\delta - \frac{3C_0\sigma}{\sqrt{n}}}n-h} } \geq \frac{1}{n}}\\
&\leq n\E \br{ \sum_{0\leq k\leq \lceil n\epsilon/h\rceil} \br{ \frac{1}{n}\sum_{j\in[n]}  \indic{\sigma \abs{W_{j\cdot} z_{s_k}^{(-j)}} > \br{1-\epsilon-4\delta - \frac{3C_0\sigma}{\sqrt{n}}}n-h} }}\\
& =  \sum_{0\leq k\leq \lceil n\epsilon/h\rceil} \br{ \sum_{j\in[n]}  \pbr{\sigma \abs{W_{j\cdot} z_{s_k}^{(-j)}} > \br{1-\epsilon-4\delta - \frac{3C_0\sigma}{\sqrt{n}}}n-h} },
\end{align*}
 Then by the same simplification as used in the derivation of (\ref{eqn:thm4_proof_4}), we have
\begin{align*}
&\pbr{\ell_m(\hat V^{\BM,m},\hat z^\MLE)\indic{\norm{W}\leq C_0\sqrt{n}} >0} \leq  n  \left\lceil \frac{n\epsilon}{h}\right\rceil\Phi\br{ \br{1-\epsilon-4\delta - \frac{3C_0\sigma}{\sqrt{n}} - \frac{h}{n}}\frac{\sqrt{n}}{\sigma }}.
\end{align*}
Recall that $\epsilon =\br{\frac{8C_0\sigma}{\sqrt{n}}}^\frac{1}{2}$, $\delta =49 \br{\frac{C_0\sigma}{\sqrt{n}}}^\frac{1}{2}$, and $h=\delta\sqrt{n}$, we have
\begin{align*}
&\pbr{\ell_m(\hat V^{\BM,m},\hat z^\MLE)\indic{\norm{W}\leq C_0\sqrt{n}} >0} \\
&\leq  n  \left\lceil \frac{2\sqrt{2n}}{49}\right\rceil\Phi\br{ \br{1-(196+2\sqrt{2})\br{\frac{C_0\sigma}{\sqrt{n}}}^\frac{1}{2}  - \frac{3C_0\sigma}{\sqrt{n}} - \frac{49}{\sqrt{n}} \br{\frac{C_0\sigma}{\sqrt{n}}}^\frac{1}{2}}\frac{\sqrt{n}}{\sigma }}\\
 &\leq n  \left\lceil \frac{2\sqrt{2n}}{49}\right\rceil\Phi\br{ \br{1-\br{196+2\sqrt{2} + \frac{49}{\sqrt{n}}}\br{\frac{C_0\sigma}{\sqrt{n}}}^\frac{1}{2} - \frac{3C_0\sigma}{\sqrt{n}} }\frac{\sqrt{n}}{\sigma }}.
\end{align*}
Similar to (\ref{eqn:thm4_proof_2}), there exists some constant $C_2>0$ that only depends on $C_0$ such that if $\frac{n}{\sigma^2}\geq C_2$, we have
\begin{align*}
\br{196+2\sqrt{2} + \frac{49}{\sqrt{n}}}\br{\frac{C_0\sigma}{\sqrt{n}}}^\frac{1}{2} + \frac{3C_0\sigma}{\sqrt{n}} \leq \frac{1}{4}.
\end{align*}
Then we have
\begin{align*}
&\pbr{\ell_m(\hat V^{\BM,m},\hat z^\MLE)\indic{\norm{W}\leq C_0\sqrt{n}} >0} \leq  n  \left\lceil \frac{2\sqrt{2n}}{49}\right\rceil\Phi\br{\frac{3\sqrt{n}}{4\sigma}} \\
&\leq n  \left\lceil \frac{2\sqrt{2n}}{49}\right\rceil \frac{2}{\sqrt{\pi}}\ebr{-\frac{9n}{32\sigma^2}}\leq \frac{1}{2}n^\frac{3}{2}\ebr{-\frac{9n}{32\sigma^2}}\leq \frac{1}{2}n^{-1},
\end{align*}
where the last inequality holds under the assumption $\frac{n}{\sigma^2}\geq 9\log n$.
Hence,
\begin{align*}
&\pbr{\ell_m(\hat V^{\BM,m},\hat z^\MLE) >0} \\
 &\leq \pbr{\ell_m(\hat V^{\BM,m},\hat z^\MLE)\indic{\norm{W}\leq C_0\sqrt{n}} >0}  +\pbr{\ell_m(\hat V^{\BM,m},\hat z^\MLE)\indic{\norm{W}> C_0\sqrt{n}} >0}  \\
  &\leq \pbr{\ell_m(\hat V^{\BM,m},\hat z^\MLE)\indic{\norm{W}\leq C_0\sqrt{n}} >0}  +\pbr{\norm{W}> C_0\sqrt{n}}\\
  &\leq  \frac{1}{2}n^{-1} +n^{-10}\leq n^{-1}.
\end{align*}

\end{proof}

\bibliographystyle{plain}
\bibliography{tightness}

\newpage
\thispagestyle{empty}
\setcounter{page}{1}
\begin{center}
\uppercase{\large Supplementary Material: Tightness of SDP and Burer-Monteiro Factorization for Phase Synchronization in High-Noise Regime}
\medskip

{BY Anderson Ye Zhang}
\medskip

{University of Pennsylvania}
\end{center}

\setcounter{section}{0}
\renewcommand{\thesection}{SM\arabic{section}}

%

\section{Proofs of Lemmas in Section \ref{sec:2}}\label{sec:proof_2}
We defer the proof of Lemma  \ref{lem:lm_contraction} to Section \ref{sec:lemmas} as the lemma is a direct generalization of Lemma 12 of \cite{zhong2018near} and our proof follows theirs.

\begin{proof}[Proof of Lemma \ref{lem:normalization}]
To prove (\ref{eqn:lem_normalization_1}), let $\theta\in[0,\pi]$ be the angle between $x$ and $y$. By  the cosine formula of triangles,  we have $\|x - y\|^2=\|x\|^2 +\|y\|^2 - 2\|x\|\|y\|\cos(\theta)$ and  $ \|x/\|x\| - y/\|y\|\|^2= 2-2\cos(\theta)$.  Consider the following scenarios. 
\begin{itemize}
\item If $\|x\|,\|y\| \geq t$, since $\|x\|^2 + \|y\|^2 \geq 2\|x\|\|y\|$, we have
\begin{align*}
\|x - y\|^2 &\geq 2\|x\|\|y\| (1- \cos(\theta))\geq 2t^2(1- \cos(\theta)) = t^2 \|x/\|x\| - y/\|y\|\|^2.
\end{align*}
Hence, $\|x/\|x\| - y/\|y\|\| \leq \|x-y\|/t$.
\item If $\|y\|\geq t >\|x\|$ and $\cos(\theta)\geq 0$, define a function $f(a,b)=a^2 + b^2 -2ab\cos(\theta)$ for $a,b\in\mathr$. Note that for any $1\geq a>0,b\geq 1$, we have $f(a,b)\geq 1-\cos^2(\theta)$. This is because $f(a,b)\geq  \min_{b'\geq 1} f(a,b') = f(a,1) = a^2 + 1 -2a\cos(\theta) \geq \min_{1\geq a' >0} f(a',1) =f(\cos(\theta),1)= 1-\cos^2(\theta)$. Hence, 
\begin{align*}
\frac{2\norm{x-y}^2}{t^2} &= 2\br{\br{\frac{\norm{x}}{t}}^2 + \br{\frac{\norm{y}}{t}}^2 - \frac{\norm{x}}{t}\frac{\norm{y}}{t}\cos(\theta)} \\
& \geq 2(1-\cos^2(\theta))\\
&\geq 2(1-\cos(\theta)) \\
& = \norm{\frac{x}{\norm{x}} - \frac{y}{\norm{y}}}^2.
\end{align*}
Hence, $\|x/\|x\| - y/\|y\|\| \leq \sqrt{2}\norm{x-y} /t$.
\item If $\|y\|\geq t >\|x\|$ and $\cos(\theta)< 0$, we have $\norm{x-y}^2 \geq \norm{y}^2 \geq t^2$ and $\|x/\|x\| - y/\|y\|\| \leq 2$. Hence, $\|x/\|x\| - y/\|y\|\| \leq 2\norm{x-y}/t$.
\item If $\norm{y}< t$, we have $\|x/\|x\| - y/\|y\|\| \leq 2= 2\indic{\norm{y}< t}$.
\end{itemize}
The proof of (\ref{eqn:lem_normalization_1}) is complete.

To prove (\ref{eqn:lem_normalization_2}), we only need to consider scenarios $x=0$ or $y=0$, as otherwise (\ref{eqn:lem_normalization_2}) is reduced to (\ref{eqn:lem_normalization_1}). If $y=0$, we have
\begin{align*}
&\norm{\br{\frac{x}{\norm{x}}\indic{x\neq 0} + u\indic{x=0}} - \br{ \frac{y}{\norm{y}}\indic{y\neq 0} + v\indic{y=0}}  } \\
&= \norm{\br{\frac{x}{\norm{x}}\indic{x\neq 0} + u\indic{x=0}}  - v}\leq 2=2\indic{\norm{y}<t}.
\end{align*}
If $x=0$ and $y\neq 0$, we have 
\begin{align*}
&\norm{\br{\frac{x}{\norm{x}}\indic{x\neq 0} + u\indic{x=0}} - \br{ \frac{y}{\norm{y}}\indic{y\neq 0} + v\indic{y=0}}  } \\
& = \norm{u - \frac{y}{\norm{y}}} \leq 2 = 2\indic{\norm{y}\geq t} + 2\indic{\norm{y}<t} = 2\indic{\norm{x-y}\geq t} + 2\indic{\norm{y}<t}\\
&\leq \frac{2\norm{x-y}}{t} + 2\indic{\norm{y}<t}.
\end{align*}
The proof of (\ref{eqn:lem_normalization_2}) is complete.
\end{proof}

\begin{proof}[Proof of Lemma \ref{lem:MLE_SDP_Loose_bound}]
Consider any $m\in\mathbb{N}\setminus\{1\}$. For simplicity, we write $\hat Z^{\BM,m}$ as $\wh{Z}$ so that $\wh{Z} =(\hat V^{\BM,m})^{\H}\hat V^{\BM,m}$. 

First, we are going to show
\begin{align}\label{eqn:MLE_SDP_Loose_bound_proof_1}
\ell(\hat V^{\BM,m},z^*) \leq \frac{4}{n^2}\Tr(z^*z^{*\H}(z^*z^{*\H}-\wh{Z})).
\end{align}
Define $b=n^{-1}\sum_{j=1}^n\hat V^{\BM,m}_j{z}_j^* = n^{-1} \hat V^{\BM,m} z^*\in\mathc^m$. If $b=0$, we have
\begin{align*}
\Tr(z^*z^{*\H}(z^*z^{*\H}-\wh{Z})) &= \Tr(z^* (z^*)^\H z^* (z^*)^\H) - \Tr(z^*z^{*\H}(\hat V^{\BM,m})^{\H}\hat V^{\BM,m})\\
& = n\Tr(z^*z^{*\H}) - \Tr(z^*(nb)^\H \hat V^{\BM,m}) \\
& = n^2.
\end{align*}
Note that $\ell(\hat V^{\BM,m},z^*) \leq n^{-1}\sum_{j\in[n]}4 = 4$. Then (\ref{eqn:MLE_SDP_Loose_bound_proof_1}) holds. In the following, we assume $b\neq 0.$
From Lemma \ref{lem:normalization}, we have for any $x,y\in\mathc^m$ such that $x\neq 0$ and $\norm{y}=1$, $\norm{x/\norm{x}-y}\leq 2\norm{x-y}$. 
Hence, we have
\begin{eqnarray*}
\ell(\hat V^{\BM,m},z^*) &=& \min_{a\in\mathbb{C}^n:\|a\|^2=1}\frac{1}{n}\sum_{j=1}^n\|\hat V^{\BM,m}_j{z}_j^*-a\|^2 \\
&=& \min_{a\in\mathbb{C}^n\backslash\{0\}}\frac{1}{n}\sum_{j=1}^n\|\hat V^{\BM,m}_j{z}_j^*-a/\|a\|\|^2 \\
&\leq& \min_{a\in\mathbb{C}^n\backslash\{0\}}\frac{4}{n}\sum_{j=1}^n\|\hat V^{\BM,m}_j{z}_j^*-a\|^2.
\end{eqnarray*}
Since the minimum of the above display is achieved when $a$ is the arithmetic mean of $\{\hat V^{\BM,m}_j{z}_j^*\}_{j\in[n]}$, i.e., $b$, we have
\begin{eqnarray*} 
\ell(\hat V^{\BM,m},z^*) &\leq & \frac{4}{n}\sum_{j=1}^n\|\hat V^{\BM,m}_j{z}_j^*-b\|^2 \\
&=& \frac{2}{n^2}\sum_{j=1}^n\sum_{l=1}^n\left(\|\hat V^{\BM,m}_j{z}_j^*-b\|^2+\|\hat V^{\BM,m}_l{z}_l^*-b\|^2\right) \\
&=& \frac{2}{n^2}\sum_{j=1}^n\sum_{l=1}^n\|\hat V^{\BM,m}_j{z}_j^*-\hat V^{\BM,m}_l{z}_l^*\|^2 \\
&=& \frac{4}{n^2}\sum_{j=1}^n\sum_{l=1}^n(1-\bar{z}_j^*{z}_l^*(\hat V^{\BM,m}_j)^{\H}\hat V^{\BM,m}_l) \\
&=& \frac{4}{n^2}\Tr(z^*z^{*\H}(z^*z^{*\H}-\wh{Z})).
\end{eqnarray*}
Therefore, (\ref{eqn:MLE_SDP_Loose_bound_proof_1}) holds.

Now it remains to upper bound $\Tr(z^*z^{*\H}(z^*z^{*\H}-\wh{Z}))$. By the definition (\ref{eqn:BM_m_def}), we have $\Tr(Y\wh{Z})\geq \Tr(Yz^*z^{*\H})$. Rearranging this inequality, we obtain $\Tr(Y(\wh{Z}-z^*z^{*\H}))\geq 0$. With (\ref{eqn:phase_model_matrix}), we have
\begin{align*}
\Tr(z^*z^{*\H}(z^*z^{*\H}-\wh{Z})) &\leq \Tr\left((Y -z^*z^{*\H})(\wh{Z}-z^*z^{*\H})\right)\\
&= \sigma\Tr\left( W(\wh{Z}-z^*z^{*\H})\right)\\
&\leq \sigma \abs{\Tr\left( W\wh{Z}\right)} +  \sigma \abs{\Tr\left( Wz^*z^{*\H}\right)} \\
& \leq  \sigma \norm{W} \Tr\left(\wh{Z}\right) + \sigma \norm{W} \Tr\left(z^*z^{*\H}\right)\\
&= 2n\sigma \norm{W}.
\end{align*}
Here, the last inequality is due to the following facts. For any two matrices $A,B\in\mathc^{n\times n}$, $\Tr(AB) \leq \norm{A} \norm{B}_*$, where $\norm{B}_*$ is the nuclear norm of $B$ that is equal to the summation of all its singular values. If $B$ is further assumed to be positive semi-definite, we have $\norm{B}_* = \Tr(B)$. In our setting, $\hat Z$ is positive semi-definite as $\min_{u\in\mathc^n} u^H \hat Z u =\min_{u\in\mathc^n} u^\H (\hat V^{\BM,m})^{\H}\hat V^{\BM,m} u \geq 0$, and so is $z^* (z^*)^\H$.

Consequently, we have $\ell(\hat V^{\BM,m},z^*)\leq \frac{8\sigma \norm{W}}{n}$. The upper bound for $\ell_1(\hat z^\MLE, z^*)$ can be established following the same steps as above and hence its proof is omitted.
\end{proof}

\section{Proofs of Lemmas in Section \ref{sec:step1}}
\begin{proof}[Proof of Lemma \ref{lem:gt}]
Consider the following scenarios. If $|x|,|y| \leq t$, we have $\abs{g_t(x) - g_t(y)}  = \frac{|x-y|}{t}$ by definition.
 If $|x|,|y| \geq t$, then 
 \begin{align*}
\abs{g_t(x) - g_t(y)}  &= \abs{\frac{x}{|x|} - \frac{y}{|y|}},
\end{align*}
Let $\theta\in[0,\pi]$ be the angle between $x$ and $y$ on the complex plane. By  the cosine formula of triangles,  we have $|x - y|^2=|x|^2 +|y|^2 - 2|x||y|\cos(\theta)$ and  $ \abs{g_t(x) - g_t(y)}^2= 2-2\cos(\theta)$. Since $|x|^2 + |y|^2 \geq 2|x||y|$, we have
\begin{align*}
|x - y|^2 &\geq 2|x||y| (1- \cos(\theta))\geq 2t^2(1- \cos(\theta)) = t^2\abs{g_t(x) - g_t(y)}^2,
\end{align*}
which yields the desired result.
If $|x|\geq t >|y|$, then
\begin{align*}
\abs{g_t(x) - g_t(y)}  &= \abs{\frac{x}{|x|} - \frac{y}{t}}.%
\end{align*}
By using the cosine formula  again, we have $\abs{g_t(x) - g_t(y)}^2 = 1+\frac{|y|^2}{t^2} - 2\frac{|y|}{t}\cos(\theta)$ and $\abs{\frac{x}{t} - \frac{y}{t}}^2 = \frac{|x|^2}{t^2} + \frac{|y|^2}{t^2} - 2\frac{|x||y|}{t^2}\cos(\theta)$. Then,
\begin{align*}
\frac{|x-y|^2}{t^2} - \abs{g_t(x) - g_t(y)}^2 & = \abs{\frac{x}{t} - \frac{y}{t}}^2  - \abs{g_t(x) - g_t(y)}^2\\
&= \frac{|x|^2}{t^2} - 1 - 2\frac{|x||y|}{t^2}\cos(\theta) + 2\frac{|y|}{t}\cos(\theta)\\
&= \br{\frac{|x|}{t}-1}\br{\frac{|x|}{t}+1} - 2 \br{\frac{|x|}{t}-1}\frac{|y|}{t}\cos(\theta)\\
&= \br{\frac{|x|}{t}-1} \br{\frac{|x|}{t}+1 - 2\frac{|y|}{t}\cos(\theta)} \\
&\geq 0,
\end{align*}
where the last inequality is due to that $\frac{|x|}{t}\geq 1 >\frac{|y|}{t}\geq 0$ and $\cos(\theta)\leq 1$. The scenario $|y|\geq t >|x|$ can be proved similarly.
\end{proof}

\begin{proof}[Proof of Lemma \ref{lem:G_propoerties}]
We prove the properties sequentially.
\begin{enumerate}
\item Recall the definition of $G$ in (\ref{eqn:G_def}).
For any $j\in[n]$, by Lemma \ref{lem:gt}, we have
\begin{align*}
\abs{[G(x,s,t)]_j - [G(y,s,t)]_j}  &= \abs{g_t(z_j^* s + \sigma[Wx]_j)  - g_t(z_j^* s + \sigma[Wy]_j)} \\
&\leq t^{-1} \abs{\br{z_j^* s + \sigma[Wx]_j} - \br{z_j^* s + \sigma[Wx]_j}} \\
& = t^{-1} \sigma\abs{[W(x-y)]_j}.
\end{align*}
Summing over all $j\in[n]$, we have
\begin{align*}
\norm{G(x,s,t) - G(y,s,t)}^2 & \leq \sum_{j\in[n]}\abs{[G(x,s,t)]_j - [G(y,s,t)]_j}^2\\
&\leq t^{-2}\sigma^2 \sum_{j\in[n]} \abs{[W(x-y)]_j}^2 \\
& = t^{-2}\sigma^2 \norm{W(x-y)}^2\\
&\leq t^{-2}\sigma^2 \norm{W}^2 \norm{x-y}^2.
\end{align*}
\item Using the first property, for any $T\in\mathbb{N}$, we have
\begin{align*}
\norm{z^{(T+1)} - z^{(T)}} &= \norm{G(z^{(T)},s,t)-G(z^{(T-1)},s,t)} \\
&\leq t^{-1}\sigma\norm{W} \norm{z^{(T)} - z^{(T-1)}}\\
&\leq \frac{1}{2}\norm{z^{(T)} - z^{(T-1)}},
\end{align*}
where the last inequality is due to the assumption $t\geq 2\sigma\norm{W}$.
\item Consider the sequence $z^{(0)}=z^*$ and  $z^{(T)} = G(z^{(T-1)},s,t)$ for all $T\in\mathbb{N}$. By the second property, we have $\norm{z^{(T+1)} - z^{(T)}} \leq \frac{1}{2}\norm{z^{(T)} - z^{(T-1)}}$ for all $T\in\mathbb{N}$. Note that $\{z^{(T)}\}$ is a sequence in $\mathcleq^n$, a complete metric space under $\norm{\cdot}$. Hence, the sequence converges to a limit $z^{(\infty)}\in\mathcleq^n$ which satisfies $z^{(\infty)} = G(z^{(\infty)},s,t)$. Hence, $z^{(\infty)}$ is a fixed point of $G(\cdot,s,t)$. Now we have proved the existence of the fixed point. To prove the uniqueness, note that if there exists another $z'\in\mathcleq^n$ such that $z' = G(z',s,t)$, we have
\begin{align*}
\norm{z^{(\infty)}-z'} = \norm{G(z^{(\infty)},s,t)- G(z',s,t)}\leq t^{-1}\sigma \norm{z^{(\infty)}-z'} \leq \norm{z^{(\infty)}-z'}/2,
\end{align*}
by the first property. Hence, $\norm{z^{(\infty)}-z'} =0$ which means $z^{(\infty)}=z'$.
\item For any $j\in[n]$, we have
\begin{align*}
\abs{[z^* s + \sigma Wz]_j - [z^* s' + \sigma Wz']_j} &\leq \abs{z^*_j s- z^*_j s'} + \sigma \abs{[W(z-z')]_j} \\
&\leq \abs{s-s'} + \sigma \abs{[W(z-z')]_j}.
\end{align*}
Summing over all $j\in[n]$, we have
\begin{align*}
\norm{\br{z^* s + \sigma Wz} - \br{z^* s' + \sigma Wz'}}^2 &\leq \sum_{j\in[n]}\br{\abs{s-s'} + \sigma \abs{[W(z-z')]_j}}^2\\
&\leq \sum_{j\in[n]} \br{2\abs{s-s'}^2 + 2 \sigma^2 \abs{[W(z-z')]_j}^2}\\
&\leq 2n\abs{s-s'}^2 + 2 \sigma^2  \norm{W}^2 \norm{z-z'}^2. \numberthis\label{lem:G_propoerties_proof_1}
\end{align*}

Note that for any $j\in[n]$, we have $z_j = [G(z,s,t)]_j=g_t([z^* s + \sigma Wz]_j)$ and similarly $z'_j = g_t([z^* s' + \sigma Wz']_j)$. Hence, by Lemma \ref{lem:gt}, we have
\begin{align*}
\abs{z_j-z'_j} &\leq t^{-1}\abs{[z^* s + \sigma Wz]_j - [z^* s' + \sigma Wz']_j}.
\end{align*}
Summing over all $j\in[n]$, by (\ref{lem:G_propoerties_proof_1}), we have
\begin{align*}
\norm{z-z'}^2&\leq t^{-2}\norm{\br{z^* s + \sigma Wz} - \br{z^* s' + \sigma Wz'}}^2 \\
&\leq 2nt^{-2}\abs{s-s'}^2 + 2 \sigma^2 t^{-2}  \norm{W}^2 \norm{z-z'}^2\\
&\leq 2nt^{-2}\abs{s-s'}^2+ \frac{1}{2}\norm{z-z'}^2,
\end{align*}
where the last inequality is due to the assumption $t\geq 2\sigma\norm{W}$. After rearrangement, we have $\norm{z-z'}^2\leq 4nt^{-2}\abs{s-s'}^2$. From (\ref{lem:G_propoerties_proof_1}), we have
\begin{align*}
\norm{\br{z^* s + \sigma Wz} - \br{z^* s' + \sigma Wz'}}^2 &\leq  2n\abs{s-s'}^2 + 2 \sigma^2  \norm{W}^2 \br{4nt^{-2}\abs{s-s'}^2}\\
&\leq 4n\abs{s-s'}^2,
\end{align*}
where the last inequality is by $t\geq 2\sigma\norm{W}$.

\end{enumerate}
\end{proof}

\begin{proof}[Proof of Lemma \ref{lem:G_approximation_error}]
Consider any $j\in[n]$. If $|z^*_j s + \sigma [Wz]_j|\geq t$, we have $[G(z,s,t)]_{j} =g_t(z^*_j s + \sigma [Wz]_j) = (z^*_j s + \sigma [Wz]_j)/|z^*_j s + \sigma [Wz]_j| = [F_1'(z,s)]_j$. If $|z^*_j s + \sigma [Wz]_j|\geq t$ is not satisfied, we have $|[F_1'(z,s)]_j| =1$ and $|[G(z,s,t)]_{j} |\leq 1$. Hence,
\begin{align*}
\abs{[F_1'(z,s)]_j - [G(z,s,t)]_{j}} &= \abs{[F_1'(z,s)]_j - [G(z,s,t)]_{j}} \indic{|z^*_j s + \sigma [Wz]_j| < t}\\
&\leq 2 \indic{|z^*_j s + \sigma [Wz]_j| < t}.
\end{align*}
Summing over all $j\in[n]$, we have
\begin{align*}
\norm{F_1'(z, s) - G(z,s,t)}^2 &= \sum_{j\in[n]}\abs{[F_1'(z,s)]_j - [G(z,s,t)]_{j}} ^2 \leq  4 \sum_{j\in[n]}\indic{|z^*_j s + \sigma [Wz]_j| < t}.
\end{align*}
\end{proof}

\begin{proof}[Proof of Lemma \ref{lem:MLE_G_fixed_points}]
Consider any $t>0$. From (\ref{eqn:MLE_F_1_prime_fixed}), we have $\hat z^\MLE = F'_1(\hat z^\MLE,\hat s)$.
Then
\begin{align*}
\norm{\hat z^\MLE - z} &= \norm{F_1'(\hat z^\MLE,\hat s) - G(z,\hat s,t)} \\
&\leq   \norm{F_1'(\hat z^\MLE,\hat s) -  F_1'(z,\hat s)} + \norm{ F_1'(z,\hat s) - G(z,\hat s,t)}\\
&\leq \norm{F_1'(\hat z^\MLE,\hat s) -  F_1'(z,\hat s)} + \sqrt{4 \sum_{j\in[n]}\indic{|z^*_j \hat s + \sigma [Wz]_j| < t}},
\end{align*}
where the last inequality is due to Lemma \ref{lem:G_approximation_error}. Hence,
\begin{align*}
\norm{\hat z^\MLE - z}^2  \leq  2\norm{F_1'(\hat z^\MLE,\hat s) -  F_1'(z,\hat s)}^2 + 8 \sum_{j\in[n]}\indic{|z^*_j \hat s + \sigma [Wz]_j| < t}.
\end{align*}

Recall the definition of $F_1'$ in (\ref{eqn:F_1_prime_def}). For any $j\in[n]$, note that $[F_1'(\hat z^\MLE,\hat s)]_j$ and $[F_1'(z,\hat s)]_j $ can be written as
\begin{align*}
[F_1'(\hat z^\MLE,\hat s)]_j &= \frac{ z^*_j \hat s + \sigma [W \hat z^\MLE]_j}{| z^*_j \hat s + \sigma [W\hat z^\MLE]_j|} \indic{ z^*_j \hat s + \sigma [W\hat z^\MLE]_j \neq 0} + \hat z^\MLE_j \indic{ z^*_j \hat s + \sigma [W\hat z^\MLE]_j = 0},\\
[F_1'(z,\hat s)]_j &= \frac{ z^*_j \hat s + \sigma [Wz]_j}{| z^*_j \hat s + \sigma [Wz]_j|} \indic{ z^*_j \hat s + \sigma [Wz]_j \neq 0} + z_j \indic{ z^*_j \hat s + \sigma [Wz]_j = 0}.
\end{align*}
By applying (\ref{eqn:lem_normalization_2}) of Lemma \ref{lem:normalization},
\begin{align*}
 \abs{[F_1'(\hat z^\MLE,\hat s)]_j -  [F_1'(z,\hat s)]_j}
 &\leq \frac{2 \abs{\br{ z^*_j \hat s + \sigma [W \hat z^\MLE]_j} - \br{ z^*_j \hat s + \sigma [Wz]_j}}}{t} + 2\indic{\abs{ z^*_j \hat s + \sigma [Wz]_j}<t} \\
 &\leq \frac{2\sigma\abs{[W( \hat z^\MLE-z)]_j}}{t}+ 2\indic{\abs{ z^*_j \hat s + \sigma [Wz]_j}<t}.
\end{align*}
Summing over all $j\in[n]$, we have 
\begin{align*}
\norm{F_1'(\hat z^\MLE,\hat s) -  F_1'(z,\hat s)}^2 &= \sum_{j\in[n]} \abs{[F_1'(\hat z^\MLE,\hat s)]_j -  [F_1'(z,\hat s)]_j}^2\\
&= \sum_{j\in[n]} \br{\frac{4\sigma^2\abs{[W( \hat z^\MLE-z)]_j}^2}{t^2} + 4\indic{\abs{ z^*_j \hat s + \sigma [Wz]_j}<t}}\\
& =  \frac{4\sigma^2}{t^2}  \norm{W ( \hat z^\MLE-z)}^2  + 4  \sum_{j\in[n]}  \indic{\abs{ z^*_j \hat s + \sigma [Wz]_j}<t}\\
&\leq  \frac{4\sigma^2}{t^2} \norm{W}^2 \norm{ \hat z^\MLE-z}^2 + 4  \sum_{j\in[n]}  \indic{\abs{ z^*_j \hat s + \sigma [Wz]_j}<t}.
\end{align*}
Hence,
\begin{align*}
\norm{\hat z^\MLE - z}^2 &\leq 2\br{\frac{4\sigma^2}{t^2} \norm{W}^2 \norm{ \hat z^\MLE-z}^2 + 4  \sum_{j\in[n]}  \indic{\abs{ z^*_j \hat s + \sigma [Wz]_j}<t}}  \\
&\quad +  8 \sum_{j\in[n]}\indic{|z^*_j \hat s + \sigma [Wz]_j| < t}\\
& =  \frac{8\sigma^2}{t^2} \norm{W}^2 \norm{ \hat z^\MLE-z}^2 + 16  \sum_{j\in[n]}  \indic{\abs{ z^*_j \hat s + \sigma [Wz]_j}<t}.
\end{align*}
When $t\geq 4\sigma \norm{W}$, we have $ \frac{8\sigma^2}{t^2} \norm{W}^2 \leq 1/2$ and the above display leads to
$\norm{\hat z^\MLE - z}^2 \leq 32\sum_{j\in[n]}  \indic{\abs{ z^*_j \hat s + \sigma [Wz]_j}<t}.$
\end{proof}

\begin{proof}[Proof of Lemma \ref{lem:MLE_extreme_points_upper_bound_2}]
Recall the definitions of $G$ in (\ref{eqn:G_def}) and $g_t$ in (\ref{eqn:gt_def}). Note that for any $t>0,a\in\mathc_1,x\in\mathc$, we have $ag_t(x) = g_t(ax)$. Hence, for any $z\in\mathc_{\leq 1}^n,s\in\mathc,t>0$, $a\in\mathc_1$, and $j\in[n]$, we have $a[G(z,s,t)]_j = ag_t([z^*s+\sigma Wz]_j) = g_t(a[z^*s+\sigma Wz]_j) = g_t([z^*(as)+\sigma W(az)]_j)$. As a result, 
\begin{align*}
\text{if $z=G(z,s,t)$, then $az=G(az,as,t)$.}
\end{align*}
This means that a fixed point of $G(\cdot,s,t)$ is also a fixed point of $G(\cdot,as,t)$.

Recall the definition of $\hat s$ in (\ref{eqn:hat_s_def}). We only need to study the case that $\hat s\neq 0$ as otherwise $G(\cdot,|\hat s|,\cdot) = G(\cdot, \hat s,\cdot )$ and Lemma \ref{lem:MLE_extreme_points_upper_bound_2} is identical to Lemma \ref{lem:MLE_extreme_points_upper_bound}. Since $\hat s\neq 0$, $\hat s/|\hat s|\in\mathc_1$ is well-defined.
For any $\delta\geq \frac{2\sigma \norm{W}}{n}$, let $z\in\mathcleq^n$ be the fixed point of $G(\cdot,|\hat s|,2\delta n)$. Then we have $\frac{\hat s }{|\hat s|} z\in\mathcleq^n$ and 
\begin{align*}
\frac{\hat s}{|\hat s|} z = G\br{\frac{\hat s}{|\hat s|} z, \frac{\hat s}{|\hat s|}|\hat s| ,2\delta n} = G\br{\frac{\hat s}{|\hat s|} z, \hat s ,2\delta n}.
\end{align*}
That is, $\frac{\hat s}{|\hat s|} z$ is the fixed point of $G(\cdot,\hat s,2\delta n)$.
By Lemma \ref{lem:MLE_extreme_points_upper_bound}, we have
\begin{align*}
\frac{1}{n}\sum_{j\in[n]}\indic{\abs{[Y\hat z^\MLE]_j}<\delta n} &\leq    \frac{9}{n}\sum_{j\in[n]}  \indic{\abs{z_j^*\hat s + \sigma \left[ W\frac{\hat s}{|\hat s|} z \right]_j} < 2 \delta n}\\
&=   \frac{9}{n}\sum_{j\in[n]}  \indic{\abs{\frac{\hat s}{|\hat s|} \br{z_j^*|\hat s| + \sigma \left[ W z \right]_j}} < 2 \delta n}\\
&=   \frac{9}{n}\sum_{j\in[n]}  \indic{\abs{z_j^*|\hat s| + \sigma \left[ W z \right]_j} < 2 \delta n}.
\end{align*}
\end{proof}

\section{Proofs of Lemmas in Section \ref{sec:step2}}
The following lemma is a counterpart of Lemma \ref{lem:G_propoerties} but for $G^{(-j)}$ instead of $G$. Then Lemma \ref{lem:G_G_j_fixed_points} is the direct consequence of the third properties of Lemmas \ref{lem:G_propoerties} and \ref{lem:G_j_properties}.
\begin{lemma}\label{lem:G_j_properties}
Consider any $j\in[n]$. The function $G^{(-j)}(\cdot,\cdot,\cdot)$ has the following properties:
\begin{enumerate}
\item For any $x,y\in\mathc^n$ and for any $s\in\mathc,t>0$, we have
\begin{align*}
\norm{G^{(-j)}(x,s,t) - G^{(-j)}(y,s,t)} &\leq  t^{-1}\sigma \norm{W} \norm{x-y}^2.
\end{align*}
\item For any $s \in\mathc,t\geq 2\sigma\norm{W}$, and for any $z^{(0,-j)}\in\mathcleq^n$, define $z^{(T,-j)} = G^{(-j)}(z^{(T-1,-j)},s,t)$ for all $T\in\mathbb{N}$. Then
\begin{align*}
\norm{z^{(T+1,-j)} - z^{(T,-j)}} &\leq \frac{1}{2}\norm{z^{(T,-j)} - z^{(T-1,-j)}}, \forall T\in\mathbb{N}.
\end{align*}
\item For any $s \in\mathc,t\geq 2\sigma\norm{W}$, $G(\cdot,s,t)$ has exactly one fixed point. That is, there exists one and only one $z\in\mathcleq^n$ such that $z= G^{(-j)}(z,s,t)$. In addition, $z$ can be achieved by iteratively applying $G^{(-j)}(\cdot,s,t)$ starting from $z^*$. That is, let $z^{(0,-j)}=z^*$ and define $z^{(T,-j)} = G^{(-j)}(z^{(T-1,-j)},s,t)$ for all $T\in\mathbb{N}$. We have $z=\lim_{T\rightarrow\infty}G^{(-j)}(z^{(T,-j)},s,t)$.
\end{enumerate} 
\end{lemma}
\begin{proof}
Note that $\norm{W^{(-j)}}\leq \norm{W}$ since $W^{(-j)}$ is obtained from $W$ by zeroing out the $j$th row and column. With this, the lemma can be proved following the exact same argument as in the proof of Lemma \ref{lem:G_propoerties}, and hence is omitted here.
\end{proof}

\begin{proof}[Proof of Lemma \ref{lem:leave_one_out_close}]
Consider any $T\in\mathbb{N}$. For any $k\in[n]$, by Lemma \ref{lem:gt}, we have
\begin{align*}
&\abs{z^{(T)}_k - z^{(T,-j)}_k}\\
&=\abs{[G(z^{(T-1)},s,t)]_k - [G^{(-j)}(z^{(T-1,-j)},s,t)]_k} \\
&= \abs{g_t([z^*s+\sigma Wz^{(T-1)}]_k) - g_t([z^*s +\sigma W^{(-j)} z^{(T-1,-j)}]_k)}\\
&\leq t^{-1} \abs{[z^*s+\sigma Wz^{(T-1)}]_k - [z^*s +\sigma W^{(-j)} z^{(T-1,-j)}]_k}\\
& = t^{-1} \sigma \abs{[ Wz^{(T-1)}]_k - [W^{(-j)} z^{(T-1,-j)}]_k} \\
& = t^{-1} \sigma \abs{[ Wz^{(T-1)}]_k - [ W^{(-j)}z^{(T-1)}]_k  + [ W^{(-j)}z^{(T-1)}]_k  - [W^{(-j)} z^{(T-1,-j)}]_k} \\
& =  t^{-1} \sigma \abs{[(W-W^{(-j)})z^{(T-1)}]_k  + [ W^{(-j)}(z^{(T-1)}-z^{(T-1,-j)})]_k }  \\
&\leq  t^{-1} \sigma\abs{[(W-W^{(-j)})z^{(T-1)}]_k} +  t^{-1} \sigma \abs{ [ W^{(-j)}(z^{(T-1)}-z^{(T-1,-j)})]_k }.
\end{align*}
If $k\neq j$, we have $[(W-W^{(-j)})z^{(T-1)}]_k = W_{kj}z^{(j)}_j$. Then the above display becomes
\begin{align*}
\abs{z^{(T)}_k - z^{(T,-j)}_k}  &\leq   t^{-1} \sigma \abs{W_{kj}z^{(j)}_j} + t^{-1} \sigma \abs{ [ W^{(-j)}(z^{(T-1)}-z^{(T-1,-j)})]_k }\\
&\leq  t^{-1} \sigma \abs{W_{kj}}+ t^{-1} \sigma \abs{ [ W^{(-j)}(z^{(T-1)}-z^{(T-1,-j)})]_k },
\end{align*}
where in the last inequality we use $|z^{(j)}_j|\leq 1$ as $z^{(j)}\in\mathc_{\leq 1}^n$.
Summing over all $k\in[n]$ such that $k\neq j$, we have
\begin{align*}
&\sum_{k\in[n]:k\neq j}\abs{z^{(T)}_k - z^{(T,-j)}_k}^2\\
&\leq \sum_{k\in[n]:k\neq j} \br{2t^{-2}\sigma^2 \abs{W_{kj}}^2 +2t^{-2}\sigma^2\abs{ [ W^{(-j)}(z^{(T-1)}-z^{(T-1,-j)})]_k ^2  }}\\
&\leq  \sum_{k\in[n]} \br{2t^{-2}\sigma^2 \abs{W_{kj}}^2 +2t^{-2}\sigma^2\abs{ [ W^{(-j)}(z^{(T-1)}-z^{(T-1,-j)})]_k ^2  }}\\
& = 2t^{-2}\sigma^2 \norm{W_j}^2 +  2t^{-2}\sigma^2 \norm{W^{(-j)}(z^{(T-1)}-z^{(T-1,-j)}}^2\\
&\leq 2t^{-2}\sigma^2 \norm{W_j}^2 +  2t^{-2}\sigma^2  \norm{W^{(-j)}}^2 \norm{z^{(T-1)}-z^{(T-1,-j)}}^2 \\
&\leq 2t^{-2}\sigma^2 \norm{W_j}^2 +  2t^{-2}\sigma^2  \norm{W}^2 \norm{z^{(T-1)}-z^{(T-1,-j)}}^2,
\end{align*}
where in  the last inequality, $\norm{W^{(-j)}} \leq \norm{W}$  due to that $W^{(-j)}$ is obtained from $W$ by zeroing out its $j$th row and column. On the other hand, $\abs{z^{(T)}_j - z^{(T,-j)}_j} \leq 2$. Hence,
\begin{align*}
\norm{z^{(T)} - z^{(T,-j)}}^2 &\leq 4 + \sum_{k\in[n]:k\neq j}\abs{z^{(T)}_k - z^{(T,-j)}_k}^2\\
&\leq 4 + 2t^{-2}\sigma^2 \norm{W_j}^2 +  2t^{-2}\sigma^2  \norm{W}^2 \norm{z^{(T-1)}-z^{(T-1,-j)}}^2\\
&\leq 4 + 2t^{-2}\sigma^2\norm{W}^2 +  2t^{-2}\sigma^2  \norm{W}^2 \norm{z^{(T-1)}-z^{(T-1,-j)}}^2,
\end{align*}
where  in the last inequality we use a fact that the operator norm of matrix is greater or equal to the norm of each column. 
When $t\geq 2\sigma \norm{W}$, we have $2t^{-2}\sigma^2  \norm{W}^2 \leq 1/2$ and
\begin{align*}
\norm{z^{(T)} - z^{(T,-j)}}^2&\leq \frac{9}{2} + \frac{1}{2}\norm{z^{(T-1)}-z^{(T-1,-j)}}^2.
\end{align*}
Note that $\norm{z^{(0)} - z^{(0,-j)}}^2=0$, by mathematical induction, it is easy to verify
$\norm{z^{(T)} - z^{(T,-j)}}^2 \leq 9,\forall T\in\mathbb{N}.$
Let $T\rightarrow\infty$, we have $\norm{z - z^{(-j)}}^2 \leq 9$.
\end{proof}

\section{Proofs of Lemmas in Section \ref{sec:exponential}}
\begin{proof}[Proof of Lemma \ref{lem:deterministic}]
From Corollary \ref{cor:fixed_points_MLE}, we have 
\begin{align*}
 \ell_m(\hat V^{\BM,m},\hat z^\MLE)\leq   \frac{8}{n}  \sum_{j\in[n]}\indic{\abs{[Y\hat z^\MLE]_j}<\delta n}.
\end{align*}
For each $k=0,1,2,\ldots, \lceil n\epsilon/h\rceil$, let $z_{s_k}\in\mathcleq^n$ be the fixed point of $G(\cdot,s_k,2\delta n)$.
Then by Corollary \ref{cor:MLE_extreme_points_upper_bound_grid}, we have
\begin{align*}
& \ell_m(\hat V^{\BM,m},\hat z^\MLE)\\
 &\leq   72\sum_{0\leq k\leq \lceil n\epsilon/h\rceil} \br{ \frac{1}{n}\sum_{j\in[n]}  \indic{\sigma \abs{  \left[ W z_{s_k} \right]_j} > s_k- 4 \delta n} }  +  \frac{72h^2}{\delta^2 n^2}\indic{h>\delta\sqrt{n}}.
\end{align*}
Since $2\delta n >2\sigma \norm{W}$, for each $k=0,1,2,\ldots, \lceil n\epsilon/h\rceil$, Proposition \ref{prop:extreme_points_leave_one_out} can be applied, leading to
\begin{align*}
 \frac{1}{n}\sum_{j\in[n]}  \indic{\sigma \abs{  \left[ W z_{s_k} \right]_j} > s_k- 4 \delta n}&\leq  \frac{1}{n}\sum_{j\in[n]}  \indic{\sigma \abs{ W_{j\cdot} z_{s_k}^{(-j)} } > s_k- 4 \delta n - 3\sigma \norm{W}}\\
 &\leq   \frac{1}{n}\sum_{j\in[n]}  \indic{\sigma \abs{ W_{j\cdot} z_{s_k}^{(-j)} } > (1-\epsilon)n - h- 4 \delta n - 3\sigma \norm{W}}\\
 & =  \frac{1}{n}\sum_{j\in[n]}  \indic{\sigma \abs{ W_{j\cdot} z_{s_k}^{(-j)} } > \br{1-\epsilon -4\delta - \frac{3\sigma \norm{W}}{n} }n - h},
\end{align*}
where in the last inequality, we use $\min_{0\leq k\leq  \lceil n\epsilon/h\rceil} s_k \geq n-(n\epsilon/h + 1)h = (1-\epsilon)n -h$.
Hence, we have
\begin{align*}
& \ell_m(\hat V^{\BM,m},\hat z^\MLE)\\
 &\leq 72\sum_{0\leq k\leq \lceil n\epsilon/h\rceil} \br{ \frac{1}{n}\sum_{j\in[n]}  \indic{\sigma \abs{ W_{j\cdot} z_{s_k}^{(-j)} } > \br{1-\epsilon -4\delta - \frac{3\sigma \norm{W}}{n} }n - h} }  +  \frac{72h^2}{\delta^2 n^2}\indic{h>\delta\sqrt{n}}.
\end{align*}
\end{proof}

\section{Auxiliary Lemmas and Proofs}\label{sec:lemmas}
The following lemma is a generalization of Lemma 11 of \cite{gao2022sdp}.
\begin{lemma}\label{lem:loss}
Consider any $m\in\mathbb{N}\setminus\{1\}$. For any $V\in\mathcal{V}_m$ and any $z\in\mathc_1^n$, we have 
\begin{align*}
\frac{1}{n^2}\fnorm{V^\H V - zz^\H}^2 \leq 2\ell_m(V,z).
\end{align*}
\end{lemma}
\begin{proof}
Lemma 11 of \cite{gao2022sdp} only considers the case where $m=n$. However, its proof holds for any $m\geq 2$, which we include here for completeness. By definition, we have
\begin{align*}
&\ell_m(V,z) = 2-\max_{a\in\mathbb{C}^n:\|a\|^2=1}\Bigg(a^{\H}\left(\frac{1}{n}\sum_{j=1}^n{z}_jV_j\right)  +\left(\frac{1}{n}\sum_{j=1}^n{z}_jV_j\right)^{\H}a\Bigg)= 2\left(1-\left\|\frac{1}{n}\sum_{j=1}^n{z}_jV_j\right\|\right).
\end{align*}
In addition, we have
\begin{eqnarray*}
n^{-2}\fnorm{{V}^{\H}V-zz^{\H}}^2 &=& \frac{1}{n^2}\sum_{j=1}^n\sum_{l=1}^n|{V}_j^{\H}{V}_l-z_j\bar{z}_l|^2 \\
& \leq  & \frac{1}{n^2}\sum_{j=1}^n\sum_{l=1}^n\left(2-{V}_j^{\H}{V}_l\bar{z}_jz_l-{V}_l^{\H}{V}_jz_j\bar{z}_l\right) \\
&=& 2\left(1-\left\|\frac{1}{n}\sum_{j=1}^n{z}_jV_j\right\|^2\right).
\end{eqnarray*}
Therefore, $n^{-2}\fnorm{{V}^{\H}V-zz^{\H}}^2\leq \ell_m(V,z)\left(2-\frac{1}{2}\ell_m(V,z)\right)\leq 2\ell_m(V,z)$, and the proof is complete.
\end{proof}

\begin{proof}[Proof of Lemma \ref{lem:lm_contraction}] We follow the proof of Lemma 12 of \cite{zhong2018near}.
We first decompose $V$ and $z$ into orthogonal components:
\begin{align}\label{eqn:lm_contraction_proof_0}
V = a(z^*)^\H + \sqrt{n}A \text{ and }z=bz^* + \sqrt{n}\beta,
\end{align}
where $a\in\mathc^m, A\in\mathc^{m\times n},b\in\mathc,\beta\in\mathc^n$ and $Az^* = 0,\beta^\H z^*=0$. Note the decomposition on $V$ is always possible as $V = V z^* (z^*)^\H + V(I_n - z^* (z^*)^\H)$ and $a= V z^*, \sqrt{n}A = V(I_n - z^* (z^*)^\H)$. By the definition of the loss $\ell_m$ in (\ref{eqn:lm_def}), there exists some $d\in\mathc^m$ such $\norm{d}=1$ and $\ell_m(V,z) =n^{-1}\fnorm{V - dz^\H}^2$. 
With the decomposition (\ref{eqn:lm_contraction_proof_0}), it means
\begin{align*}
n\ell_m(V,z)  &=\fnorm{V - dz^\H}^2\\
&= \fnorm{\br{a(z^*)^\H + \sqrt{n}A} - d\br{bz^* + \sqrt{n}\beta}^\H}^2\\
&= \fnorm{\br{a - d\bar{b}}(z^*)^\H + \sqrt{n}(A - d\beta^\H)}^2\\
&=  \fnorm{\br{a - d\bar{b}}(z^*)^\H}^2 + \fnorm{\sqrt{n}(A - d\beta^\H)}^2\\
&= n \norm{a-d\bar{b}}^2 + n \fnorm{A - d\beta^\H}^2. \numberthis \label{eqn:lm_contraction_proof_3}
\end{align*}
where the third equation is due to the orthogonality $(A - d\beta^H) z^* =0$. Then
\begin{align}\label{eqn:lm_contraction_proof_4}
 \fnorm{A - d\beta^\H} &\leq \sqrt{\ell_m(V,z)}.
\end{align}
We also have 
\begin{align*}
 \fnorm{V Y^\H - d(Yz)^\H}&= \fnorm{V(z^*(z^*)^\H + \sigma W)^\H - dz^\H (z^*(z^*)^\H + \sigma W)^\H}\\
 &\leq \fnorm{(V - dz^\H)z^*(z^*)^\H} + \fnorm{\sigma (V - dz^\H) W}\\
 &\leq  \fnorm{(a(z^*)^\H - d\bar{b}(z^*)^\H)z^*(z^*)^\H} + \sigma \norm{W}\fnorm{V - dz^\H}\\
 &\leq n\sqrt{n}\norm{a - d\bar{b}}+ \sigma \norm{W}\sqrt{n}\sqrt{\ell_m(V,z)}, \numberthis \label{eqn:lm_contraction_proof_2}
\end{align*}
where the second inequality is due to the fact that $\fnorm{B_1B_2}\leq \fnorm{B_1}\opnorm{B_2}$ for any two matrices $B_1,B_2$.
If
\begin{align}\label{eqn:lm_contraction_proof_1}
\norm{a - d\bar{b}}\leq 6\epsilon\fnorm{A - d\beta^\H}
\end{align}
holds, (\ref{eqn:lm_contraction_proof_2}) and (\ref{eqn:lm_contraction_proof_4}) leads to
\begin{align*}
\ell_m(V Y^\H, Yz) &\leq \frac{1}{n} \fnorm{V Y^\H - d(Yz)^\H}^2\\
&\leq \frac{1}{n}\br{ 6\epsilon n \sqrt{n} \fnorm{A - d\beta^\H} +  \sigma \norm{W}\sqrt{n}\sqrt{\ell_m(V,z)}}^2\\
&\leq  \frac{1}{n}\br{ 6\epsilon n\sqrt{n} \sqrt{\ell_m(V,z)} +  \sigma \norm{W}\sqrt{n}\sqrt{\ell_m(V,z)}}^2\\
&= n^2\br{6\epsilon  + \frac{\sigma \norm{W}}{n}}^2 \ell_m(V,z),
\end{align*}
which yields the desired result. The remaining proof is devoted to establishing (\ref{eqn:lm_contraction_proof_1}).

Note that
\begin{align*}
\ell_m(V,z^*)& =\min_{u\in\mathc^m:\norm{u}=1}n^{-1}\fnorm{a(z^*)^\H + \sqrt{n}A - u(z^*)^\H}^2 \\
& = \min_{u\in\mathc^m:\norm{u}=1}n^{-1} \br{\fnorm{(a-u)(z^*)^\H}^2 +  \fnorm{ \sqrt{n}A}^2} \\
&=  \min_{u\in\mathc^m:\norm{u}=1} \norm{a-u}^2 + \fnorm{A}^2.
\end{align*}
Since $\ell_m(V,z^*)\leq \epsilon^2<1/4$, we have $ \fnorm{A}^2 \leq \epsilon^2$, $\norm{a}\neq 0$ and $\min_{u\in\mathc^m:\norm{u}=1} \norm{a-u}^2 = \norm{a - a/\norm{a}}^2 = (1-\norm{a})^2$. Together with $1 = n^{-1}\fnorm{V}^2  =  n^{-1}\norm{a(z^*)^\H}^2 +  n^{-1}\fnorm{\sqrt{n}A}^2 = \norm{a}^2 + \fnorm{A}^2$, we have 
\begin{align*}
\ell_m(V,z^*)& = (1-\norm{a})^2 + 1 - \norm{a}^2 = 2- 2\norm{a}.
\end{align*}
Then $\ell_m(V,z^*)\leq \epsilon^2$ leads to $1\geq \norm{a}\geq 1-\epsilon^2/2$. Similarly for $z$, we have $\norm{\beta}^2\leq \epsilon^2$, $1\geq \abs{b}\geq 1-\epsilon^2/2$ and $1=|b|^2 + \norm{\beta}^2$. Since $\epsilon<1/2$, we have $\norm{a}+|b|>1$, and consequently $\abs{\norm{a} - |b|} \leq \abs{\norm{a}- |b|}(\norm{a}+|b|) = |{\norm{a}^2- |b|^2}|$. Since $\norm{a}^2 + \fnorm{A}^2 = |b|^2 + \norm{\beta}^2$, we have $|\norm{a}^2 - |b|^2| =  |\norm{\beta}^2 - \fnorm{A}^2| $. Together with  $\fnorm{A}^2,\norm{\beta}^2\leq \epsilon^2$, we have
\begin{align*}
\abs{\norm{a} - |b|} &\leq  |\norm{\beta}^2 - \fnorm{A}^2| =  |\norm{\beta} - \fnorm{A}|\br{\norm{\beta} + \fnorm{A}} \\
& \leq 2\epsilon |\norm{\beta} - \fnorm{A}|\leq 2\epsilon \fnorm{A - d\beta^\H}. \numberthis\label{eqn:lm_contraction_proof_5}
\end{align*}
Note that
\begin{align*}
\norm{a-d\bar{b}} &= \norm{a - \frac{a}{\norm{a}} |b| +  \frac{a}{\norm{a}} \frac{b}{|b|} \bar{b} - d \bar{b}}\\
&\leq \norm{a - \frac{a}{\norm{a}} |b|} +\norm{\br{ \frac{a}{\norm{a}} \frac{b}{|b|} - d }\bar{b}}\\
&= \abs{\norm{a} - |b|} + \norm{\frac{a}{\norm{a}} \frac{b}{|b|}-d} |b|\\
&\leq 2\epsilon \fnorm{A - d\beta^\H} +  \norm{\frac{a}{\norm{a}} \frac{b}{|b|}-d},
\end{align*}
where in the last inequality we use $|b|\leq 1$. Hence, to establish (\ref{eqn:lm_contraction_proof_1}), we only need to show
\begin{align}\label{eqn:lm_contraction_proof_6}
\norm{\frac{a}{\norm{a}} \frac{b}{|b|}-d} \leq 4\epsilon \fnorm{A - d\beta^\H}.
\end{align}

To prove (\ref{eqn:lm_contraction_proof_6}), define $d_0 = \frac{a}{\norm{a}}\frac{{b}}{|b|}\in\mathc^m$. Then $\norm{d_0}=1$.  Similar to (\ref{eqn:lm_contraction_proof_3}), we have $\fnorm{V -d_0z^\H}^2 = n\norm{a-d_0\bar b}^2 + n\fnorm{A - d_0\beta^\H}^2$. By the definition of $d$, $\fnorm{V -d z^\H}^2\leq \fnorm{V -d_0 z^\H}^2$, which leads to
 \begin{align*}
\norm{a-d\bar{b}}^2 +  \fnorm{A - d\beta^\H}^2 \leq  \norm{a-d_0\bar b}^2 + \fnorm{A - d_0\beta^\H}^2.
 \end{align*}
 Note that $d_0 \bar b = a \frac{|b|}{\norm{a}}$ is proportional to $a$ and $\norm{d_0\bar b} =\norm{d\bar b} = |b|$. Let $\theta\in[0,\pi]$ be the angle between $a$ and $d\bar b$ in $\mathc^m$. By the cosine formula of triangles, we have 
\begin{align*}
&\norm{a -d \bar{b}}^2 = \norm{a}^2 + \norm{d\bar b}^2 - 2\norm{a}|d\bar b|\cos(\theta) = \norm{a}^2  +|b|^2 -2\norm{a}|b|\cos(\theta)  \\
&\norm{a - d_0\bar b}^2  = \norm{a - a \frac{|b|}{\norm{a}}}^2 = \norm{a}^2 +  |b|^2 - 2\norm{a}|b| \\
\text{and }&\norm{d - d_0} ^2 = \norm{d}^2 + \norm{d_0}^2 - 2\norm{d}\norm{d_0} \cos(\theta)= 2 (1-\cos(\theta)). \numberthis \label{eqn:lm_contraction_proof_7}
\end{align*}
 Hence, $\norm{a -d \bar{b}}^2 - \norm{a - d_0\bar b}^2= 2\norm{a}|b|(1 -\cos(\theta))$. By the triangle inequality,  $\fnorm{A - d_0\beta^\H} -  \fnorm{A - d\beta^\H} \leq \fnorm{(d_0 - d)\beta^\H} = \norm{d_0 -d } \norm{\beta}\leq \epsilon \norm{d_0 -d }$ where in the last inequality we use $\norm{\beta}\leq \epsilon$. Then,
 \begin{align*}
 2\norm{a}|b|(1 -\cos(\theta)) &\leq  \fnorm{A - d_0\beta^\H}^2 -   \fnorm{A - d\beta^\H}^2 \\
 &=  \br{\fnorm{A - d_0\beta^\H} -   \fnorm{A - d\beta^\H}}\br{\fnorm{A - d_0\beta^\H} - \fnorm{A - d\beta^\H} + 2  \fnorm{A - d\beta^\H}}\\
 & \leq  \epsilon \norm{d_0 -d } \br{\epsilon \norm{d_0 -d } + 2  \fnorm{A - d\beta^\H}}.
 \end{align*}
 By (\ref{eqn:lm_contraction_proof_7}), it becomes $\norm{a}|b| \norm{d_0-d}^2 \leq \epsilon \norm{d_0 -d } \br{\epsilon \norm{d_0 -d } + 2  \fnorm{A - d\beta^\H}}$, which further leads to
\begin{align*}
(\epsilon^{-1}\norm{a}|b| - \epsilon)\norm{d_0 -d } \leq 2\fnorm{A - d\beta^\H}.
\end{align*}
Since $\norm{a},|b|\geq 1-\epsilon^2/2$, we have $\epsilon^{-1}\norm{a}|b| - \epsilon \geq \epsilon^{-1}(1-\epsilon^2/2)^2 -\epsilon \geq \epsilon^{-1}(1-\epsilon^2)-\epsilon = \epsilon^{-1}(1-2\epsilon^2) >(2\epsilon)^{-1}$ where the last inequality is due to $\epsilon<1/2$. Hence, $(2\epsilon)^{-1}\norm{d_0 -d } \leq 2\fnorm{A - d\beta^\H}$, which establishes (\ref{eqn:lm_contraction_proof_6}). The proof of the lemma is complete.
\end{proof}

\end{document}